\documentclass[
final
]{dmtcs-episciences}

\usepackage[utf8]{inputenc}
\usepackage{subfigure}

\usepackage{amsmath,amssymb,url,color,tikz,colortbl,stmaryrd,bm}
\usetikzlibrary{patterns}
\usepackage{young}

\usepackage{hyperref}
\usepackage{amsfonts}
\usepackage{amsmath}
\usepackage{amssymb}
\usepackage{amsthm}
\usepackage{xcolor}

\theoremstyle{theorem}
\newtheorem{thm}{Theorem}

\newtheorem{prop}[thm]{Proposition}

\newtheorem{cor}[thm]{Corollary}

\theoremstyle{definition}
\newtheorem{defn}[thm]{Definition}
\newtheorem{definition}[thm]{Definition}

\newtheorem{notation}[thm]{Notation}
\newtheorem{ex}[thm]{Example}
\newtheorem{example}[thm]{Example}

\newtheorem{remark}[thm]{Remark}

\makeatletter
\newcommand{\raisemath}[1]{\mathpalette{\raisem@th{#1}}}
\newcommand{\raisem@th}[3]{\raisebox{#1}{$#2#3$}}
\makeatother

\usepackage{graphicx}
\usepackage{listings}
\usepackage{tram}
\usepackage{manfnt}
\usepackage{url,color,tikz,euscript}

\usepackage{scalerel}
\usepackage{mathtools}

\numberwithin{thm}{section}
\catcode`@=11
\def\m@th{\mathsurround\z@}
\def\cases#1{\left\{\,\vcenter{\normalbaselines\m@th
    \ialign{$##\hfil$&\quad##\hfil\crcr#1\crcr}}\right.}
\def\hang{\hangindent 24pt}
\def\d@nger{\medbreak\begingroup\clubpenalty=10000
  \def\par{\endgraf\endgroup\medbreak} %
  \noindent\hang\hangafter=-2
  \hbox to0pt{\hskip-\hangindent\dbend\hfill}}
\outer\def\danger{\d@nger}
\newcommand{\arxiv}[1]{\href{http://arxiv.org/abs/#1}{\texttt{arXiv:#1}}}

\newcommand{\rr}{\mathbb{R}}
\newcommand{\zz}{\mathbb{Z}}

\newcommand{\nn}{\mathbb{N}}

\newcommand{\kk}{\mathbb{K}}
\newcommand{\bbs}{\mathbb{S}}

\newcommand{\st}{\operatorname{ST}}
\newcommand{\ST}{\operatorname{ST}}

\newcommand{\ra}{\rightarrow}
\newcommand{\dra}{\dashrightarrow}
\newcommand{\down}{\nabla}
\newcommand{\up}{\Delta}
\newcommand{\BAR}{\operatorname{BAR}}

\newcommand{\BOR}{\operatorname{BOR}}

\newcommand{\NAR}{\operatorname{NAR}}

\newcommand{\sm}{\setminus}
\definecolor{green}{HTML}{006600}
\definecolor{orange}{HTML}{FF6200}
\definecolor{purple}{HTML}{990099}
\definecolor{coral}{HTML}{FF7F50}
\definecolor{mahogany}{HTML}{C04000}
\definecolor{gold}{HTML}{DAA541}
\definecolor{chocolate}{HTML}{D5691E}

\newcommand{\cala}{\mathcal{A}}

\newcommand{\calc}{\mathcal{C}}

\newcommand{\calf}{\mathcal{F}}

\newcommand{\calj}{\mathcal{J}}

\newcommand{\calo}{\mathcal{O}}
\newcommand{\calp}{\mathcal{P}}

\newcommand{\calr}{\mathcal{R}}
\newcommand{\cals}{\mathcal{S}}

\def\rowm{\rho}
\def\rowA{\rowm_{\mathcal{A}}}

\newcommand{\cblu}[1]{{\color{blue}#1}}

\newcommand{\ds}{\displaystyle}

\newcommand\mydots{\ifmmode\makebox[1.2em][c]{$\cdot$\hfil$\cdot$\hfil$\cdot$}\fi}

\author{Michael Joseph\affiliationmark{1}
  \and Tom Roby\affiliationmark{2}}
\title{A birational lifting of the Stanley--Thomas word on products of two chains}
\affiliation{
  Dalton State College, Dalton, GA, USA\\
  University of Connecticut, Storrs, CT, USA}
\keywords{antichains,
birational rowmotion,
chain polytope,
cyclic rotation,
dynamical algebraic combinatorics,
homomesy,
noncommutative algebra,
product of chains,
Stanley--Thomas word}
\received{2020-07-12}
\revised{2021-01-15}
\accepted{2021-06-03}
\begin{document}
\publicationdetails{23}{2021}{1}{17}{6633}
\maketitle
\begin{abstract}
  The dynamics of certain combinatorial actions and their liftings to actions at the
piecewise-linear and birational level have been studied lately with an eye towards questions
of periodicity, orbit structure, and invariants.  One key property enjoyed by the rowmotion
operator on certain finite partially-ordered sets is homomesy, where the average value of a
statistic is the same for all orbits.  To prove refined versions of homomesy in the product
of two chain posets, J.~Propp and the second author used an equivariant bijection discovered (less
formally) by R.~Stanley and H.~Thomas.

We explore the lifting of this ``Stanley--Thomas word'' to the piecewise-linear, birational,
and noncommutative realms.  Although the map is no longer a bijection, so cannot be used to
prove periodicity directly, it still gives enough information to prove the homomesy at the
piecewise-linear and birational levels (a result previously shown by D.~Grinberg, S.~Hopkins, and S.~Okada). Even at the
noncommutative level, the Stanley--Thomas word of a poset labeling rotates cyclically with
the lifting of antichain rowmotion.  Along the way we give some formulas for noncommutative
antichain rowmotion that we hope will be first steps towards proving the conjectured
periodicity at this level.

\end{abstract}

\section{Introduction}\label{sec:intro}

Birational liftings of combinatorial actions are a subject of active interest in algebraic
combinatorics.  There are birational versions of the Robinson--Schensted--Knuth correspondence~\cite{kirillov2001introduction,noumi-yamada}
and of the rowmotion operator on a poset.  In many cases these liftings are accomplished by
first extending the map to a piecewise-linear action on $\rr$-labelings of posets (typically
that live within a certain polytope).  From this piecewise-linear setting, we then
detropicalize to get an action on labelings of posets by rational functions. This was first
done for rowmotion of order ideals by Einstein and Propp, who also lifted some of the
homomesy properties from the combinatorial setting to these higher
levels~\cite{einpropp}.  Given an action on a set of combinatorial objects, we call a statistic
on those objects \emph{homomesic} if the average value of the statistic along every orbit
is the same~\cite{propproby}. 

Rowmotion and related operations (at the combinatorial, piecewise-linear, and birational
levels) can be realized as products of simple involutions, called
\emph{toggles}, thereby situating them within a toggle \emph{group}, whose properties can be
studied~\cite{cameronfonder}.  One way to lift actions such as rowmotion
is to simply lift the notion of toggling to the piecewise-linear and birational levels.
(This works also at the noncommutative level, though the toggles are no longer involutions.) 
Cameron and Fon-Der-Flaass's toggle group was originally for order ideals of a
poset, but the notion has been extended much more widely by Striker~\cite{strikergentog}, in
particular to toggling of antichains.  

Combinatorial rowmotion was originally studied as a map on antichains of a poset, though
it can be equivariantly considered as a map on order ideals.  In fact, one of the original
examples of homomesy was the conjecture of Panyushev, later proven by Armstrong, Stump, and
Thomas, that cardinality is homomesic for the action of antichain rowmotion on
root posets of finite-dimensional Lie algebras~\cite{panyushev,ast}.
So it was natural to consider
piecewise-linear and birational liftings of antichain rowmotion via their own toggle group.
The first author gave an equivariant bijection between the antichain toggle group and the
order-ideal one at the combinatorial and piecewise-linear levels~\cite{antichain-toggling}.  In later work we lifted
this bijection to the birational and noncommutative levels,
and constructed the birational
and noncommutative analogues of antichain rowmotion~\cite{BAR-motion}.  This allows us to
transfer some properties, such as periodicity and orbit structure, proven for one
action to the other.  

A key tool in dynamical algebraic combinatorics is the construction of equivariant
maps between actions of interest and actions which are easier to comprehend,
particularly ones that involve cyclic rotation.  When these maps are bijections, they
frequently explain most observed phenomena.  Even when these maps fail to be injective, they
still provide useful information about the action in question, as in the \emph{resonance}
phenomenon of Dilks, Pechenik, and Striker~\cite{dpsresonance, dilks-striker-vorland} and
the w-tuple of Grinberg and the second author~\cite[\S5]{GrRo16}.

In this paper we lift one such equivariant bijection for antichain rowmotion on rectangular posets
$P=[a]\times[b]$, where $[n]:= \{1,2,\dots,n\}$, called the \emph{Stanley--Thomas word}.
Each antichain corresponds to a binary string of length $a+b$, and rowmotion corresponds to
cyclically rotating the corresponding binary string.  Besides proving periodicity, this
bijection also allowed Propp and the second author to prove that ``fiber-restricted''
cardinality statistics (thus total cardinality also) were homomesic with respect to this action.  Our
lifting is no longer a bijection (so does not prove periodicity); however, it does exhibit the
corresponding homomesy properties at the piecewise-linear and birational levels.
Surprisingly, even at
the noncommutative level, it exhibits the key property of cyclically rotating equivariantly
with the lifting of antichain rowmotion to the noncommutative level.  This allows us to write
all our proofs in this realm, then specialize down to get the corresponding results at the
birational, piecewise-linear, and combinatorial levels.  

We organize the paper as follows.  In Section~\ref{sec:background} we set notation and
review necessary background regarding rowmotion, toggle groups, homomesy, and the combinatorial Stanley--Thomas
word.  In Section~\ref{sec:PLBRlifts}, we lift the Stanley--Thomas word to the piecewise-linear
and birational levels, and recall the definitions of BAR-motion (Birational Antichain
Rowmotion) and of (multiplicative) homomesy.  The main result is that the Stanley--Thomas word of a labeling
cyclically rotates equivariantly with BAR-motion acting on the labeling.  This lifts the proof of fiber homomesy for antichain rowmotion to the birational setting, a result previously written up using different means by S.~Hopkins, and also discovered independently by D.~Grinberg and S.~Okada in unpublished work~\cite[Remarks~4.44,~4.45]{hopkins2019minuscule}.  In
Section~\ref{sec:NClift} we further lift all of this to poset labelings by elements of a
\emph{skew field} $\bbs$ of characteristic zero, obtaining the analogous equivariant bijection.  Finally in
Section~\ref{sec:future}, we describe possible directions for future research, and take a
first step in one of those directions, giving explicit formulas for the first pass of BAR-motion through the poset $[a]\times [b]$. 

\section{Background}\label{sec:background}

Let $P$ be a finite poset.  An \textbf{antichain} of $P$ is a subset of (the elements of) $P$ which contains no two
comparable elements.  We denote the collection of all antichains of $P$ by $\cala (P)$. (For further background
information about posets, see~\cite[Ch.~3]{ec1ed2}.)

This paper will largely be concerned with the special (but important) poset that is a product of two chains:
$P=[a]\times [b]$, where $[n]:=\{1,2,\dots ,n \}$.
Richard Stanley and Hugh Thomas gave
a bijection $A \leftrightarrow w(A)$
between the set $\cala([a]\times[b])$ of antichains 
of $[a]\times[b]$ and the set of binary $(a+b)$-tuples
with exactly $a$ 0s and $b$ 1s. We now call $w(A)$ the Stanley--Thomas word of the antichain $A$~\cite[\S 3.3.2]{propproby}.

\begin{defn}[{\cite[remark after Thm.~2.5]{Sta09}}]\label{def:stcomb}
Fix $a,b\in \zz_{>0}$.  For $1\leq k\leq a$, the subset
$\{(k,\ell):1\leq \ell \leq b\}$ of $[a]\times[b]$ is called the \textbf{\textit{k}th positive fiber}.
For $1\leq \ell\leq b$, the subset
$\{(k,\ell):1\leq k \leq a\}$ of $[a]\times[b]$ is called the \textbf{\textit{$\ell$}th negative fiber}.
The {\bf Stanley--Thomas word} (or \textbf{ST word}) $w(A)$ of an antichain
$A\in\cala([a]\times[b])$ is the tuple
$(w_1,w_2,\dots,w_{a+b})$ given by:
$$w_i=\left\{\begin{array}{ll}
1 &\text{if $1\leq i \leq a$ and $A$ has an element in the $i$th positive fiber},\\
1 &\text{if $a+1\leq i \leq a+b$ and $A$ has NO element in the $(i-a)$th negative fiber},\\
0 &\text{otherwise}.\\
\end{array}\right.$$
\end{defn}

\begin{ex}\label{ex:SW-[3]x[5]}
Consider the poset $[3]\times[5]$ below, and the
antichain $A\in\cala([3]\times[5])$ shown on the right (where filled-in circles indicate the elements in $A$).
Since $A$ contains elements in the 2nd and 3rd positive fibers, the first 3 entries of $w(A)$ are $0,1,1$.
Since $A$ contains elements in the 1st and 4th negative fibers, the last 5 entries of $w(A)$ are
$0,1,1,0,1$.  So $w(A)=(0,1,1,0,1,1,0,1)$.
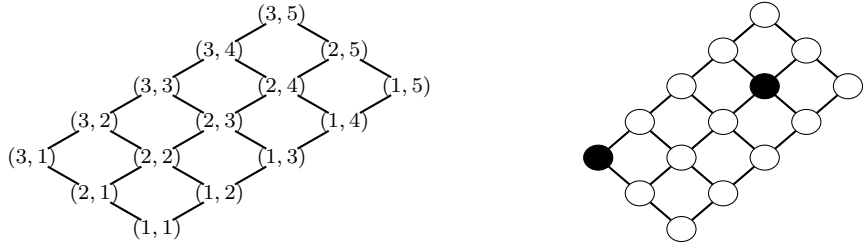
\begin{center}
\begin{footnotesize}
    \begin{tikzpicture}[xscale=.83, yscale=.475]
\node at (0,0) {$(1,1)$};
\node at (1,1) {$(1,2)$};
\node at (2,2) {$(1,3)$};
\node at (3,3) {$(1,4)$};
\node at (4,4) {$(1,5)$};
\node at (-1,1) {$(2,1)$};
\node at (0,2) {$(2,2)$};
\node at (1,3) {$(2,3)$};
\node at (2,4) {$(2,4)$};
\node at (3,5) {$(2,5)$};
\node at (-2,2) {$(3,1)$};
\node at (-1,3) {$(3,2)$};
\node at (0,4) {$(3,3)$};
\node at (1,5) {$(3,4)$};
\node at (2,6) {$(3,5)$};
\draw[thick] (0.25, 0.25) -- (0.75, 0.75);
\draw[thick] (1.25, 1.25) -- (1.75, 1.75);
\draw[thick] (2.25, 2.25) -- (2.75, 2.75);
\draw[thick] (3.25, 3.25) -- (3.75, 3.75);
\draw[thick] (-0.75, 1.25) -- (-0.25, 1.75);
\draw[thick] (0.25, 2.25) -- (0.75, 2.75);
\draw[thick] (1.25, 3.25) -- (1.75, 3.75);
\draw[thick] (2.25, 4.25) -- (2.75, 4.75);
\draw[thick] (-1.75, 2.25) -- (-1.25, 2.75);
\draw[thick] (-0.75, 3.25) -- (-0.25, 3.75);
\draw[thick] (0.25, 4.25) -- (0.75, 4.75);
\draw[thick] (1.25, 5.25) -- (1.75, 5.75);
\draw[thick] (-0.25, 0.25) -- (-0.75,0.75);
\draw[thick] (-1.25, 1.25) -- (-1.75,1.75);
\draw[thick] (0.75, 1.25) -- (0.25,1.75);
\draw[thick] (-0.25, 2.25) -- (-0.75,2.75);
\draw[thick] (1.75, 2.25) -- (1.25,2.75);
\draw[thick] (0.75, 3.25) -- (0.25,3.75);
\draw[thick] (2.75, 3.25) -- (2.25,3.75);
\draw[thick] (1.75, 4.25) -- (1.25,4.75);
\draw[thick] (3.75, 4.25) -- (3.25,4.75);
\draw[thick] (2.75, 5.25) -- (2.25,5.75);
\begin{scope}[shift={(8.4,0)},xscale=2/3]
\draw (0,0) circle [radius=0.35];
\draw (1,1) circle [radius=0.35];
\draw (2,2) circle [radius=0.35];
\draw (3,3) circle [radius=0.35];
\draw (4,4) circle [radius=0.35];
\draw (-1,1) circle [radius=0.35];
\draw (0,2) circle [radius=0.35];
\draw (1,3) circle [radius=0.35];
\draw[fill] (2,4) circle [radius=0.35];
\draw (3,5) circle [radius=0.35];
\draw[fill] (-2,2) circle [radius=0.35];
\draw (-1,3) circle [radius=0.35];
\draw (0,4) circle [radius=0.35];
\draw (1,5) circle [radius=0.35];
\draw (2,6) circle [radius=0.35];
\draw[thick] (0.25, 0.25) -- (0.75, 0.75);
\draw[thick] (1.25, 1.25) -- (1.75, 1.75);
\draw[thick] (2.25, 2.25) -- (2.75, 2.75);
\draw[thick] (3.25, 3.25) -- (3.75, 3.75);
\draw[thick] (-0.75, 1.25) -- (-0.25, 1.75);
\draw[thick] (0.25, 2.25) -- (0.75, 2.75);
\draw[thick] (1.25, 3.25) -- (1.75, 3.75);
\draw[thick] (2.25, 4.25) -- (2.75, 4.75);
\draw[thick] (-1.75, 2.25) -- (-1.25, 2.75);
\draw[thick] (-0.75, 3.25) -- (-0.25, 3.75);
\draw[thick] (0.25, 4.25) -- (0.75, 4.75);
\draw[thick] (1.25, 5.25) -- (1.75, 5.75);
\draw[thick] (-0.25, 0.25) -- (-0.75,0.75);
\draw[thick] (-1.25, 1.25) -- (-1.75,1.75);
\draw[thick] (0.75, 1.25) -- (0.25,1.75);
\draw[thick] (-0.25, 2.25) -- (-0.75,2.75);
\draw[thick] (1.75, 2.25) -- (1.25,2.75);
\draw[thick] (0.75, 3.25) -- (0.25,3.75);
\draw[thick] (2.75, 3.25) -- (2.25,3.75);
\draw[thick] (1.75, 4.25) -- (1.25,4.75);
\draw[thick] (3.75, 4.25) -- (3.25,4.75);
\draw[thick] (2.75, 5.25) -- (2.25,5.75);
\end{scope}
    \end{tikzpicture}
\end{footnotesize}
\end{center}
\end{ex}

One property of interest of the ST word is that the invertible map of
(antichain) {\bf rowmotion} $\rowA:\cala(P) \ra \cala(P)$ corresponds equivariantly to
cyclic rotation of the ST word. 
Rowmotion is a map first studied by Brouwer and Schrijver~\cite{brouwer1974period}
with several names in the literature;
the name ``rowmotion'' due to Striker and Williams~\cite{strikerwilliams} has stuck.

To define $\rowA$, we first define the sets $\calj(P)$
of order ideals of $P$ and $\calf(P)$ of order filters of $P$.
A subset $I\subseteq P$ is called an {\bf order ideal} (resp.\ {\bf order filter}) of $P$ 
if for all $x\in I$ and $y<x$ (resp.\ $y>x$) in $P$,
$y\in I$.
Using the notation of Einstein and Propp~\cite{einpropp} (Version 3, or
see~\cite{BAR-motion}), $\rowA=\down\circ \Theta \circ \up^{-1}$ where 
\begin{itemize}
\item $\Theta:\calj(P) \ra \calf(P)$ is the {\bf complementation} map given by $\Theta(I)=P\sm I$,
\item $\down: \calf(P) \ra \cala (P)$ is the {\bf down-transfer} map where $\down(F)$ is the set of minimal elements of the filter $F$,
\item $\up^{-1}: \cala (P) \ra \calj(P)$ is called {\bf downward saturation} or {\bf inverse up-transfer}.
For any antichain $A$ of $P$,
$\up^{-1}(A)=\{x\in P:x\leq y\text{ for some }y\in A\}$.
\end{itemize}

\begin{ex}\label{ex:[3]x[5]row}
For the antichain $A$ of Example~\ref{ex:SW-[3]x[5]}, the ST word is $w(A)=(0,1,1,0,1,1,0,1)$.
Below we show the effect of rowmotion on $A$ giving the antichain whose ST word is
$w\big(\rowA(A)\big)=(1,0,1,1,0,1,1,0)$, a rightward cyclic shift of $w(A)$.

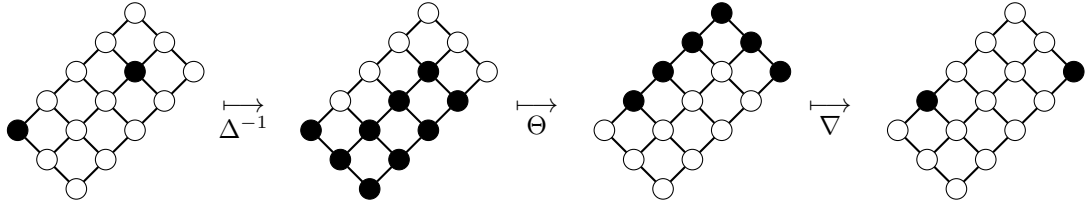
\begin{tikzpicture}[scale=.39]
\begin{scope}[shift={(0,0)}]
\draw (0,0) circle [radius=0.35];
\draw (1,1) circle [radius=0.35];
\draw (2,2) circle [radius=0.35];
\draw (3,3) circle [radius=0.35];
\draw (4,4) circle [radius=0.35];
\draw (-1,1) circle [radius=0.35];
\draw (0,2) circle [radius=0.35];
\draw (1,3) circle [radius=0.35];
\draw[fill] (2,4) circle [radius=0.35];
\draw (3,5) circle [radius=0.35];
\draw[fill] (-2,2) circle [radius=0.35];
\draw (-1,3) circle [radius=0.35];
\draw (0,4) circle [radius=0.35];
\draw (1,5) circle [radius=0.35];
\draw (2,6) circle [radius=0.35];
\draw[thick] (0.25, 0.25) -- (0.75, 0.75);
\draw[thick] (1.25, 1.25) -- (1.75, 1.75);
\draw[thick] (2.25, 2.25) -- (2.75, 2.75);
\draw[thick] (3.25, 3.25) -- (3.75, 3.75);
\draw[thick] (-0.75, 1.25) -- (-0.25, 1.75);
\draw[thick] (0.25, 2.25) -- (0.75, 2.75);
\draw[thick] (1.25, 3.25) -- (1.75, 3.75);
\draw[thick] (2.25, 4.25) -- (2.75, 4.75);
\draw[thick] (-1.75, 2.25) -- (-1.25, 2.75);
\draw[thick] (-0.75, 3.25) -- (-0.25, 3.75);
\draw[thick] (0.25, 4.25) -- (0.75, 4.75);
\draw[thick] (1.25, 5.25) -- (1.75, 5.75);
\draw[thick] (-0.25, 0.25) -- (-0.75,0.75);
\draw[thick] (-1.25, 1.25) -- (-1.75,1.75);
\draw[thick] (0.75, 1.25) -- (0.25,1.75);
\draw[thick] (-0.25, 2.25) -- (-0.75,2.75);
\draw[thick] (1.75, 2.25) -- (1.25,2.75);
\draw[thick] (0.75, 3.25) -- (0.25,3.75);
\draw[thick] (2.75, 3.25) -- (2.25,3.75);
\draw[thick] (1.75, 4.25) -- (1.25,4.75);
\draw[thick] (3.75, 4.25) -- (3.25,4.75);
\draw[thick] (2.75, 5.25) -- (2.25,5.75);
\node at (5.7,2.8) {$\longmapsto$};
\node[below] at (5.7,2.8) {$\up^{-1}$};
\end{scope}
\begin{scope}[shift={(10,0)}]
\draw[fill] (0,0) circle [radius=0.35];
\draw[fill] (1,1) circle [radius=0.35];
\draw[fill] (2,2) circle [radius=0.35];
\draw[fill] (3,3) circle [radius=0.35];
\draw (4,4) circle [radius=0.35];
\draw[fill] (-1,1) circle [radius=0.35];
\draw[fill] (0,2) circle [radius=0.35];
\draw[fill] (1,3) circle [radius=0.35];
\draw[fill] (2,4) circle [radius=0.35];
\draw (3,5) circle [radius=0.35];
\draw[fill] (-2,2) circle [radius=0.35];
\draw (-1,3) circle [radius=0.35];
\draw (0,4) circle [radius=0.35];
\draw (1,5) circle [radius=0.35];
\draw (2,6) circle [radius=0.35];
\draw[thick] (0.25, 0.25) -- (0.75, 0.75);
\draw[thick] (1.25, 1.25) -- (1.75, 1.75);
\draw[thick] (2.25, 2.25) -- (2.75, 2.75);
\draw[thick] (3.25, 3.25) -- (3.75, 3.75);
\draw[thick] (-0.75, 1.25) -- (-0.25, 1.75);
\draw[thick] (0.25, 2.25) -- (0.75, 2.75);
\draw[thick] (1.25, 3.25) -- (1.75, 3.75);
\draw[thick] (2.25, 4.25) -- (2.75, 4.75);
\draw[thick] (-1.75, 2.25) -- (-1.25, 2.75);
\draw[thick] (-0.75, 3.25) -- (-0.25, 3.75);
\draw[thick] (0.25, 4.25) -- (0.75, 4.75);
\draw[thick] (1.25, 5.25) -- (1.75, 5.75);
\draw[thick] (-0.25, 0.25) -- (-0.75,0.75);
\draw[thick] (-1.25, 1.25) -- (-1.75,1.75);
\draw[thick] (0.75, 1.25) -- (0.25,1.75);
\draw[thick] (-0.25, 2.25) -- (-0.75,2.75);
\draw[thick] (1.75, 2.25) -- (1.25,2.75);
\draw[thick] (0.75, 3.25) -- (0.25,3.75);
\draw[thick] (2.75, 3.25) -- (2.25,3.75);
\draw[thick] (1.75, 4.25) -- (1.25,4.75);
\draw[thick] (3.75, 4.25) -- (3.25,4.75);
\draw[thick] (2.75, 5.25) -- (2.25,5.75);
\node at (5.7,2.8) {$\longmapsto$};
\node[below] at (5.7,2.8) {$\Theta$};
\end{scope}
\begin{scope}[shift={(20,0)}]
\draw (0,0) circle [radius=0.35];
\draw (1,1) circle [radius=0.35];
\draw (2,2) circle [radius=0.35];
\draw (3,3) circle [radius=0.35];
\draw[fill] (4,4) circle [radius=0.35];
\draw (-1,1) circle [radius=0.35];
\draw (0,2) circle [radius=0.35];
\draw (1,3) circle [radius=0.35];
\draw (2,4) circle [radius=0.35];
\draw[fill] (3,5) circle [radius=0.35];
\draw (-2,2) circle [radius=0.35];
\draw[fill] (-1,3) circle [radius=0.35];
\draw[fill] (0,4) circle [radius=0.35];
\draw[fill] (1,5) circle [radius=0.35];
\draw[fill] (2,6) circle [radius=0.35];
\draw[thick] (0.25, 0.25) -- (0.75, 0.75);
\draw[thick] (1.25, 1.25) -- (1.75, 1.75);
\draw[thick] (2.25, 2.25) -- (2.75, 2.75);
\draw[thick] (3.25, 3.25) -- (3.75, 3.75);
\draw[thick] (-0.75, 1.25) -- (-0.25, 1.75);
\draw[thick] (0.25, 2.25) -- (0.75, 2.75);
\draw[thick] (1.25, 3.25) -- (1.75, 3.75);
\draw[thick] (2.25, 4.25) -- (2.75, 4.75);
\draw[thick] (-1.75, 2.25) -- (-1.25, 2.75);
\draw[thick] (-0.75, 3.25) -- (-0.25, 3.75);
\draw[thick] (0.25, 4.25) -- (0.75, 4.75);
\draw[thick] (1.25, 5.25) -- (1.75, 5.75);
\draw[thick] (-0.25, 0.25) -- (-0.75,0.75);
\draw[thick] (-1.25, 1.25) -- (-1.75,1.75);
\draw[thick] (0.75, 1.25) -- (0.25,1.75);
\draw[thick] (-0.25, 2.25) -- (-0.75,2.75);
\draw[thick] (1.75, 2.25) -- (1.25,2.75);
\draw[thick] (0.75, 3.25) -- (0.25,3.75);
\draw[thick] (2.75, 3.25) -- (2.25,3.75);
\draw[thick] (1.75, 4.25) -- (1.25,4.75);
\draw[thick] (3.75, 4.25) -- (3.25,4.75);
\draw[thick] (2.75, 5.25) -- (2.25,5.75);
\node at (5.7,2.8) {$\longmapsto$};
\node[below] at (5.7,2.8) {$\down$};
\end{scope}
\begin{scope}[shift={(30,0)}]
\draw (0,0) circle [radius=0.35];
\draw (1,1) circle [radius=0.35];
\draw (2,2) circle [radius=0.35];
\draw (3,3) circle [radius=0.35];
\draw[fill] (4,4) circle [radius=0.35];
\draw (-1,1) circle [radius=0.35];
\draw (0,2) circle [radius=0.35];
\draw (1,3) circle [radius=0.35];
\draw (2,4) circle [radius=0.35];
\draw (3,5) circle [radius=0.35];
\draw (-2,2) circle [radius=0.35];
\draw[fill] (-1,3) circle [radius=0.35];
\draw (0,4) circle [radius=0.35];
\draw (1,5) circle [radius=0.35];
\draw (2,6) circle [radius=0.35];
\draw[thick] (0.25, 0.25) -- (0.75, 0.75);
\draw[thick] (1.25, 1.25) -- (1.75, 1.75);
\draw[thick] (2.25, 2.25) -- (2.75, 2.75);
\draw[thick] (3.25, 3.25) -- (3.75, 3.75);
\draw[thick] (-0.75, 1.25) -- (-0.25, 1.75);
\draw[thick] (0.25, 2.25) -- (0.75, 2.75);
\draw[thick] (1.25, 3.25) -- (1.75, 3.75);
\draw[thick] (2.25, 4.25) -- (2.75, 4.75);
\draw[thick] (-1.75, 2.25) -- (-1.25, 2.75);
\draw[thick] (-0.75, 3.25) -- (-0.25, 3.75);
\draw[thick] (0.25, 4.25) -- (0.75, 4.75);
\draw[thick] (1.25, 5.25) -- (1.75, 5.75);
\draw[thick] (-0.25, 0.25) -- (-0.75,0.75);
\draw[thick] (-1.25, 1.25) -- (-1.75,1.75);
\draw[thick] (0.75, 1.25) -- (0.25,1.75);
\draw[thick] (-0.25, 2.25) -- (-0.75,2.75);
\draw[thick] (1.75, 2.25) -- (1.25,2.75);
\draw[thick] (0.75, 3.25) -- (0.25,3.75);
\draw[thick] (2.75, 3.25) -- (2.25,3.75);
\draw[thick] (1.75, 4.25) -- (1.25,4.75);
\draw[thick] (3.75, 4.25) -- (3.25,4.75);
\draw[thick] (2.75, 5.25) -- (2.25,5.75);
\end{scope}
\end{tikzpicture}
\end{ex}

Clearly, as the ST word has length $a+b$, when we shift it
$a+b$ times, we obtain $w(A)$ again.
This proves rowmotion on $\cala([a]\times[b])$ has order $a+b$.
Propp and the second author also used it to prove a homomesy result in terms of fibers.
Let $\cals$ be a finite set, and $f:\cals \rightarrow \kk$ a
``statistic'' (any map) on $\cals$, where $\kk$ is a field of characteristic 0.  We call $f$ \textbf{homomesic} with
respect to an invertible map (aka ``action'')
$\varphi :\cals \rightarrow \cals$ if the average of $f$ over every $\varphi$-orbit is the
same~\cite{propproby}.
Consider the statistics $p_i:\cala(P)\ra\rr$ and $n_i:\cala(P)\ra\rr$ where $p_i(A)$ (resp.~$n_i(A)$) is 1 if $A$ has an element
in the $i$th positive fiber (resp.~negative fiber) and 0 otherwise.
It follows from the rotation property of the ST word that
$p_i$ and $n_i$ are homomesic with average
$b/(a+b)$ for $p_i$  and $a/(a+b)$ for $n_i$ on any orbit.
As the cardinality of an antichain can be expressed as
$p_1+p_2+\cdots+p_a$, we see that cardinality on $\cala(P)$ is homomesic with average $ab/(a+b)$~\cite[\S3.3.2]{propproby}.
See Figure~\ref{fig:2x2-CAR} for an illustration of this property for $a=b=2$.

\begin{figure}[h]
\begin{center}
\begin{tikzpicture}[scale=.54]
\begin{scope}[shift={(0,-4)}]
\draw[ultra thick] (0.13, 1.87) -- (0.87, 1.13);
\draw[red,ultra thick] (1.13, 1.13) -- (1.87, 1.87);
\draw[ultra thick] (1.87, 2.13) -- (1.13, 2.87);
\draw[blue,ultra thick] (0.13, 2.13) -- (0.87, 2.87);
\draw[red] (1,1) circle [radius=0.2];
\draw[blue] (0,2) circle [radius=0.2];
\draw[red] (2,2) circle [radius=0.2];
\draw[blue] (1,3) circle [radius=0.2];
\node at (1,0) {$({\color{red}0},{\color{blue}0},1,1)$};
\node at (1,-1) {${\color{blue}0}+{\color{red}0}=0$};
\node[right] at (-4,2) {Orbit:};
\node[right] at (-4,0) {ST word:};
\node[right] at (-4,-1) {cardinality:};
\end{scope}
\node at (3.5,-1.7) {$\stackrel{\rowA}{\longmapsto}$};
\begin{scope}[shift={(5,-4)}]
\draw[ultra thick] (0.13, 1.87) -- (0.87, 1.13);
\draw[red,ultra thick] (1.13, 1.13) -- (1.87, 1.87);
\draw[ultra thick] (1.87, 2.13) -- (1.13, 2.87);
\draw[blue,ultra thick] (0.13, 2.13) -- (0.87, 2.87);
\draw[red,fill] (1,1) circle [radius=0.2];
\draw[blue] (0,2) circle [radius=0.2];
\draw[red] (2,2) circle [radius=0.2];
\draw[blue] (1,3) circle [radius=0.2];
\node at (1,0) {$({\color{red}1},{\color{blue}0},0,1)$};
\node at (1,-1) {${\color{blue}0}+{\color{red}1}=1$};
\end{scope}
\node at (8.5,-1.7) {$\stackrel{\rowA}{\longmapsto}$};
\begin{scope}[shift={(10,-4)}]
\draw[ultra thick] (0.13, 1.87) -- (0.87, 1.13);
\draw[red,ultra thick] (1.13, 1.13) -- (1.87, 1.87);
\draw[ultra thick] (1.87, 2.13) -- (1.13, 2.87);
\draw[blue,ultra thick] (0.13, 2.13) -- (0.87, 2.87);
\draw[red] (1,1) circle [radius=0.2];
\draw[blue,fill] (0,2) circle [radius=0.2];
\draw[red,fill] (2,2) circle [radius=0.2];
\draw[blue] (1,3) circle [radius=0.2];
\node at (1,0) {$({\color{red}1},{\color{blue}1},0,0)$};
\node at (1,-1) {${\color{blue}1}+{\color{red}1}=2$};
\end{scope}
\node at (13.5,-1.7) {$\stackrel{\rowA}{\longmapsto}$};
\begin{scope}[shift={(15,-4)}]
\draw[ultra thick] (0.13, 1.87) -- (0.87, 1.13);
\draw[red,ultra thick] (1.13, 1.13) -- (1.87, 1.87);
\draw[ultra thick] (1.87, 2.13) -- (1.13, 2.87);
\draw[blue,ultra thick] (0.13, 2.13) -- (0.87, 2.87);
\draw[red] (1,1) circle [radius=0.2];
\draw[blue] (0,2) circle [radius=0.2];
\draw[red] (2,2) circle [radius=0.2];
\draw[blue,fill] (1,3) circle [radius=0.2];
\node at (1,0) {$({\color{red}0},{\color{blue}1},1,0)$};
\node at (1,-1) {${\color{blue}1}+{\color{red}0}=1$};
\node at (5.5,-1) {AVG: ${\color{blue}\frac12}+{\color{red}\frac12}=1$};
\end{scope}
\node at (18.5,-1.7) {$\stackrel{\rowA}{\longmapsto}$};
\node at (20,-1.7) {\Huge $:\hspace{-0.1 in}||$};
\draw[thick] (-4,-5.5) -- (23,-5.5);
\begin{scope}[shift={(0,-9)}]
\draw[ultra thick] (0.13, 1.87) -- (0.87, 1.13);
\draw[red,ultra thick] (1.13, 1.13) -- (1.87, 1.87);
\draw[ultra thick] (1.87, 2.13) -- (1.13, 2.87);
\draw[blue,ultra thick] (0.13, 2.13) -- (0.87, 2.87);
\draw[red] (1,1) circle [radius=0.2];
\draw[blue,fill] (0,2) circle [radius=0.2];
\draw[red] (2,2) circle [radius=0.2];
\draw[blue] (1,3) circle [radius=0.2];
\node at (1,0) {$({\color{red}0},{\color{blue}1},0,1)$};
\node at (1,-1) {${\color{blue}1}+{\color{red}0}=1$};
\node[right] at (-4,2) {Orbit:};
\node[right] at (-4,0) {ST word:};
\node[right] at (-4,-1) {cardinality:};
\end{scope}
\node at (3.5,-6.7) {$\stackrel{\rowA}{\longmapsto}$};
\begin{scope}[shift={(5,-9)}]
\draw[ultra thick] (0.13, 1.87) -- (0.87, 1.13);
\draw[red,ultra thick] (1.13, 1.13) -- (1.87, 1.87);
\draw[ultra thick] (1.87, 2.13) -- (1.13, 2.87);
\draw[blue,ultra thick] (0.13, 2.13) -- (0.87, 2.87);
\draw[red] (1,1) circle [radius=0.2];
\draw[blue] (0,2) circle [radius=0.2];
\draw[red,fill] (2,2) circle [radius=0.2];
\draw[blue] (1,3) circle [radius=0.2];
\node at (1,0) {$({\color{red}1},{\color{blue}0},1,0)$};
\node at (1,-1) {${\color{blue}0}+{\color{red}1}=1$};
\node at (5.5,-1) {AVG: ${\color{blue}\frac12}+{\color{red}\frac12}=1$};
\node at (5.5,-2) {\phantom{Make caption lower.}};
\end{scope}
\node at (8.5,-6.7) {$\stackrel{\rowA}{\longmapsto}$};
\node at (10,-6.7) {\Huge $:\hspace{-0.1 in}||$};
\end{tikzpicture}\vspace{-0.3 in}
\end{center}
\caption{The two orbits of $\rowA$ on $P=[2]\times[2]$.
The symbol $:\hspace{-0.05 in}||$ means to repeat, so $\rowA$ has order 4 on $P=[2]\times[2]$.  Below each antichain is its ST word and cardinality.  The average cardinality is $\frac{2\cdot2}{4}=1$
in both orbits.  The positive fiber statistics $p_1$
(in {\color{red}red}) and $p_2$ (in {\color{blue}blue}) have average $\frac24=\frac12$ across each orbit.}
\label{fig:2x2-CAR}
\end{figure}
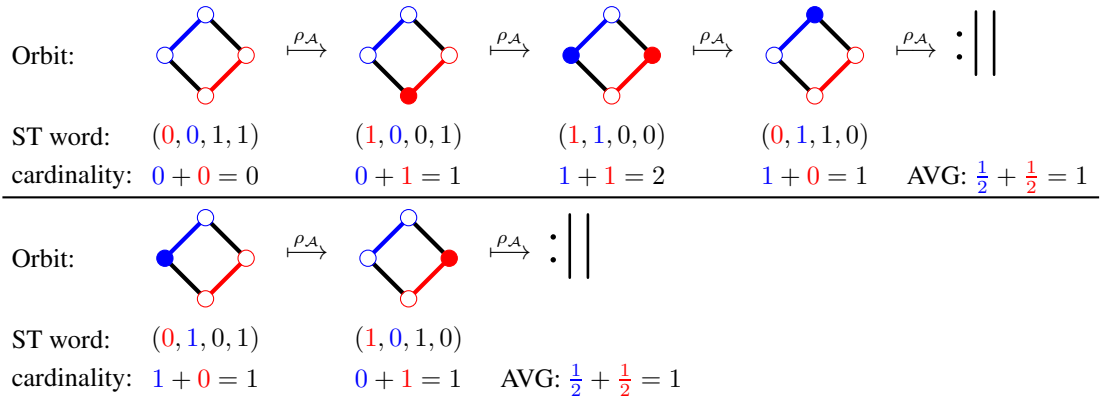

We can associate each antichain $A\in\cala([a]\times[b])$ to its indicator function
defined by: $A(i,j)$ is 1 if $(i,j)\in A$ and 0 if $(i,j)\not\in A$.
Then the ST word $w(A)$ has the following alternate description, since no fiber can contain multiple elements of $A$.

\begin{prop}\label{prop:SW-alternate}
Given $a,b\in\zz_{>0}$, $P=[a]\times[b]$, and
$A\in\cala(P)$, the $i$th entry of $w(A)$ is
$$w_i=\left\{\begin{array}{ll}
\sum\limits_{j=1}^b A(i,j) &\text{if $1\leq i \leq a$},\\
1 - \sum\limits_{j=1}^a A(j,i-a) &\text{if $a+1\leq i \leq a+b$}.\\
\end{array}\right.$$
\end{prop}

\begin{remark}
The convention in previous literature has been to use $-1$ in the definition of $w(A)$ in
place of 0.  At the combinatorial level, this difference is insignificant since it does not
change the key cyclic rotation property of the ST word under rowmotion. 
Using 0 allows us to use the conceptually simple expressions $\sum_{j=1}^b A(i,j)$ and $1-\sum_{j=1}^a A(j,i-a)$ 
in Proposition~\ref{prop:SW-alternate}, as opposed to $2\sum_{j=1}^b A(i,j) - 1$ and
$1-2\sum_{j=1}^a A(j,i-a)$ for the $-1$ convention.

\end{remark}

\section{Lifting to the
piecewise-linear and birational realms}\label{sec:PLBRlifts}

Einstein and Propp gave the first generalizations of rowmotion to the piecewise
linear and birational settings, lifting order-ideal rowmotion to an action on Stanley's
order polytope, thence to labelings of $P$ by rational functions~\cite{einpropp}.  The
parallel lifting of antichain rowmotion was done in ~\cite{antichain-toggling, BAR-motion}.
We just include the definitions and basic outline needed here; see the above papers for more details.  

\begin{definition}[\cite{Sta86}]
Let $\rr^{P}$ denote the set of labelings of the finite poset $P$ by real numbers.  
Within $\rr^{P}$ the \textbf{order polytope}
 of $P$
 is the set $\calo\calp(P)$ 
of labelings $f:P\ra [0,1]$ that are order-preserving: if $a\leq b$ in $P$, then $f(a)\leq f(b)$.
Similarly the \textbf{order-reversing polytope}
of $P$ is the set $\calo\calr(P)$ of labelings
$f: P \ra [0,1]$ that are order-reversing:
if $a\leq b$ in $P$, then $f(a)\geq f(b)$.
The \textbf{chain polytope} of $P$ is the set $\calc(P)$ of labelings $f:P\ra [0,1]$ such that the sum of the labels
across every chain is at most 1.  
\end{definition}
By associating a subset of $P$ with its indicator function, the sets $\calf(P)$, $\calj(P)$, and $\cala(P)$ describe the vertices of $\calo\calp(P)$, $\calo\calr(P)$, and $\calc(P)$ respectively~\cite{Sta86}.

There are eight kinds of rowmotion:
order and antichain, each
considered at the combinatorial, piecewise-linear, birational, or noncommutative levels.  
Each can be described in two
different ways.  One is as composition of three maps that generalize complementation,
down-transfer (combinatorially, taking minimal elements of an order filter), and downward
saturation.   The other is as a ``Coxeter element'' of appropriate toggles performed once at
each element of $P$ along a linear
extension~\cite{cameronfonder,einpropp,antichain-toggling,BAR-motion}. 
For simplicity we just give the former
definitions of these various rowmotions via the three-step compositions; however, the
toggling definitions are easily accessible in our earlier papers.

\begin{defn}[{\cite[\S4]{einpropp}}]\label{def:trans-PL}
The maps $\Theta: \rr^P \ra \rr^P$, $\down: \calo\calp(P) \ra \calc(P)$,
$\up: \calo\calr(P) \ra \calc(P)$, and their inverses are given as follows.
To ensure that we are never taking the maximum of an empty set,
we extend $P$ to the poset $\widehat{P}$
by adjoining a minimal element $\widehat{0}$ and
maximal element $\widehat{1}$,
with $\widehat{0}<x<\widehat{1}$
for all $x\in P$.  If $x$ and $y$ in $P$ satisfy $x<y$ and there is no $z\in P$ such that
$x<z<y$, then we say \textbf{$y$ covers $x$} or \textbf{$x$ is covered by $y$} and write
$x\lessdot y$ or $y\gtrdot x$. 

For all $f\in \kk^P$ and $x\in P$ we set:
\begin{align*}
    (\Theta f)(x) &= 1-f(x),\\
    (\down f)(x) &= f(x)-{\max\limits_{y\lessdot x} f(y)}\  \cblu{\big(\text{with } f\big(\widehat{0}\big)=0\big)}, \\
    (\up f)(x) &= 
    f(x) - {\max\limits_{y\gtrdot x} f(y)}\  \cblu{\big(\text{with } f\big(\widehat{1}\big)=0\big)},\\
    \left(\down^{-1}f\right)(x) &=
    \max\left\{ f(y_1)+f(y_2)+\cdots +f(y_k):
    \widehat{0}\lessdot y_1\lessdot y_2 \lessdot \cdots \lessdot y_k=x \right\}\\
    &=f(x)+ \max\limits_{y\lessdot x}\left(\down^{-1}f\right)(y) \cblu{\;\; \big(\text{with } (\down^{-1}f)\big(\widehat{0}\big)=0\big)},\\
    \left(\up^{-1}f\right)(x) &=
    \max\left\{ f(y_1)+ f(y_2)+\cdots +f(y_k):x=y_1
\lessdot y_2\lessdot \cdots \lessdot y_k\lessdot \widehat{1}\right\}\\
    &=f(x)+ \max\limits_{y\gtrdot x}\left(\up^{-1}f\right)(y) \;\; \cblu{\big(\text{with } (\up^{-1}f)\big(\widehat{1}\big)=0\big)}. 
\end{align*}
\end{defn}


We use the same symbols in each realm (combinatorial, piecewise-linear, birational, and noncommutative), allowing
context to clarify which is meant. Using Proposition~\ref{prop:SW-alternate},
the Stanley--Thomas word naturally generalizes from $\cala(P)$ to $\calc(P)$. (Here we have set the $\alpha$ and $\omega$
of \cite[\S4]{einpropp} to 0 and 1, respectively.  The maps $\nabla^{-1}$ and $\Delta^{-1}$ were formerly denoted as
$\mathbf{OP}$ and $\mathbf{OR}$ in \cite{antichain-toggling}.)  

\begin{defn}[\cite{einpropp,antichain-toggling}]
\label{def:rho-PL}
In the piecewise-linear setting,
we define \textbf{PL antichain rowmotion}
(or \textbf{chain-polytope rowmotion}) by
$\rho_\calc= \down\circ\Theta\circ\up^{-1}: \calc (P)\ra \calc (P)$.  
\end{defn}

There is also PL \emph{order} rowmotion (or \emph{order-polytope rowmotion}) defined as
$\Theta\circ\up^{-1}\circ\down: \calo\calp(P)\ra \calo\calp(P)$.
While the order rowmotion maps have received more attention thus far in the dynamical algebraic combinatorics community,
we will only use the antichain rowmotion maps in this paper.  Note that the order and antichain perspectives are related
through equivariance.

\begin{defn}\label{def:ST-PL}
Let $a,b\in\zz_{>0}$, $P=[a]\times[b]$, and
$g\in\calc(P)$.
The piecewise-linear \textbf{Stanley--Thomas word} (or \textbf{ST word}) $\st_g$ is the $(a+b)$-tuple
whose $i$th entry is given by
$$\st_g(i)=\left\{\begin{array}{ll}
\sum\limits_{j=1}^b g(i,j) &\text{if $1\leq i \leq a$},\\
1 - \sum\limits_{j=1}^a g(j,i-a) &\text{if $a+1\leq i \leq a+b$}.\\
\end{array}\right.$$
\end{defn}

\begin{remark}
Note that when $g$ is the indicator function of an antichain $A$, we get $\st_{g} = w(A)$,
the original ST word from Definition~\ref{def:stcomb}.
One fundamental difference between the combinatorial ST word on $\cala(P)$
and the piecewise-linear analogue on $\calc(P)$
is that an element $g\in\calc(P)$
is not uniquely determined from $\st_g$.  
For example, in $\calc([2]\times[2])$, both of the
following labelings have the same ST word:
$(0.6,0.5,0.7,0.2)$.
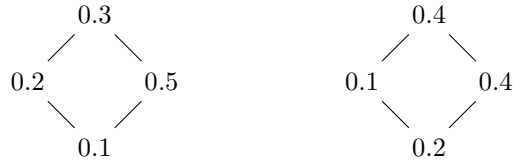
\begin{center}
    \begin{tikzpicture}[scale=8/9]
    \node at (0,0) {$0.1$};
    \node at (-1,1) {$0.2$};
    \node at (1,1) {$0.5$};
    \node at (0,2) {$0.3$};
    \draw (-0.3,0.3) -- (-0.7,0.7);
    \draw (0.3,0.3) -- (0.7,0.7);
    \draw (-0.3,1.7) -- (-0.7,1.3);
    \draw (0.3,1.7) -- (0.7,1.3);
    \node at (5,0) {$0.2$};
    \node at (4,1) {$0.1$};
    \node at (6,1) {$0.4$};
    \node at (5,2) {$0.4$};
    \draw (4.7,0.3) -- (4.3,0.7);
    \draw (5.3,0.3) -- (5.7,0.7);
    \draw (4.7,1.7) -- (4.3,1.3);
    \draw (5.3,1.7) -- (5.7,1.3);
    \end{tikzpicture}
\end{center}
\end{remark}

The Stanley--Thomas word rotates equivariantly with the action of chain-polytope rowmotion on the
corresponding labeling and allows us to derive some refined homomesies, analogous to the
combinatorial case.
See Figure~\ref{fig:2x2-PLAR} for a sample orbit.
Proofs will follow by tropicalizing their birational analogues, which we now construct. 

\begin{figure}
\begin{center}
\begin{small}
\begin{tikzpicture}[scale=.51]
\begin{scope}[shift={(1,-3)}]
\draw[blue,ultra thick] (-0.3, 1.7) -- (-0.7, 1.3);
\draw[ultra thick] (0.3, 1.7) -- (0.7, 1.3);
\draw[ultra thick] (-0.7, 0.7) -- (-0.3, 0.3);
\draw[red, ultra thick] (0.7, 0.7) -- (0.3, 0.3);
\node at (0,2) {$  {\color{blue} 0.3}  $};
\node at (-1,1) {$ {\color{blue} 0.1}  $};
\node at (1,1) {$  {\color{red} 0.4}  $};
\node at (0,0) {$  {\color{red} 0.2}  $};
\node at (0,-1) {$({\color{red}0.6},{\color{blue}0.4},0.7,0.3)$};
\node at (0,-2) {${\color{blue}0.4}+{\color{red}0.6}=1$};
\node[right] at (-5.7,1) {Orbit:};
\node[right] at (-5.7,-1) {ST word:};
\node[right] at (-5.7,-2) {label sum:};
\end{scope}
\node at (3.5,-1.7) {$\stackrel{\rho_{\calc}}{\longmapsto}$};
\begin{scope}[shift={(6,-3)}]
\draw[blue,ultra thick] (-0.3, 1.7) -- (-0.7, 1.3);
\draw[ultra thick] (0.3, 1.7) -- (0.7, 1.3);
\draw[ultra thick] (-0.7, 0.7) -- (-0.3, 0.3);
\draw[red, ultra thick] (0.7, 0.7) -- (0.3, 0.3);
\node at (0,2) {$  {\color{blue} 0.1}  $};
\node at (-1,1) {$ {\color{blue} 0.5}  $};
\node at (1,1) {$  {\color{red} 0.2}  $};
\node at (0,0) {$  {\color{red} 0.1}  $};
\node at (0,-1) {$({\color{red}0.3},{\color{blue}0.6},0.4,0.7)$};
\node at (0,-2) {${\color{blue}0.6}+{\color{red}0.3}=0.9$};
\end{scope}
\node at (8.5,-1.7) {$\stackrel{\rho_{\calc}}{\longmapsto}$};
\begin{scope}[shift={(11,-3)}]
\draw[blue,ultra thick] (-0.3, 1.7) -- (-0.7, 1.3);
\draw[ultra thick] (0.3, 1.7) -- (0.7, 1.3);
\draw[ultra thick] (-0.7, 0.7) -- (-0.3, 0.3);
\draw[red, ultra thick] (0.7, 0.7) -- (0.3, 0.3);
\node at (0,2) {$  {\color{blue} 0.2}  $};
\node at (-1,1) {$ {\color{blue} 0.1}  $};
\node at (1,1) {$  {\color{red} 0.4}  $};
\node at (0,0) {$  {\color{red} 0.3}  $};
\node at (0,-1) {$({\color{red}0.7},{\color{blue}0.3},0.6,0.4)$};
\node at (0,-2) {${\color{blue}0.3}+{\color{red}0.7}=1$};
\end{scope}
\node at (13.5,-1.7) {$\stackrel{\rho_{\calc}}{\longmapsto}$};
\begin{scope}[shift={(16,-3)}]
\draw[blue,ultra thick] (-0.3, 1.7) -- (-0.7, 1.3);
\draw[ultra thick] (0.3, 1.7) -- (0.7, 1.3);
\draw[ultra thick] (-0.7, 0.7) -- (-0.3, 0.3);
\draw[red, ultra thick] (0.7, 0.7) -- (0.3, 0.3);
\node at (0,2) {$  {\color{blue} 0.1}  $};
\node at (-1,1) {$ {\color{blue} 0.6}  $};
\node at (1,1) {$  {\color{red} 0.3}  $};
\node at (0,0) {$  {\color{red} 0.1}  $};
\node at (0,-1) {$({\color{red}0.4},{\color{blue}0.7},0.3,0.6)$};
\node at (0,-2) {${\color{blue}0.7}+{\color{red}0.4}=1.1$};
\node at (5.5,-2) {AVG: ${\color{blue}0.5}+{\color{red}0.5}=1$};
\node at (5.5,-2) {\phantom{Make caption lower.}};
\end{scope}
\node at (18.5,-1.7) {$\stackrel{\rho_{\calc}}{\longmapsto}$};
\node at (20,-1.7) {\Huge $:\hspace{-0.1 in}||$};
\end{tikzpicture}
\end{small}
\end{center}
\caption{One orbit of chain-polytope rowmotion on $P=[2]\times[2]$.
The label sum is the analogue of cardinality in the piecewise-linear realm.
The positive fiber statistics $p_1$
(in {\color{red}red}) and $p_2$ (in {\color{blue}blue}) have average $0.5$ across each orbit.}
\label{fig:2x2-PLAR}
\end{figure}
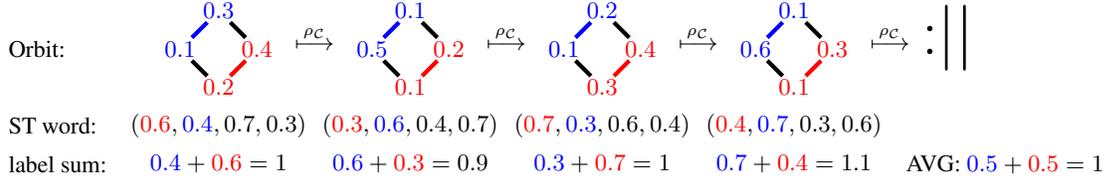

Let $\kk$ be a field of characteristic zero, and let $\kk^P$ denote the set of labelings
$f:P\ra \kk $ of the elements of $P$ by elements of $\kk$.  
To lift piecewise-linear maps to birational maps over the field $\kk$, we detropicalize by replacing
the $\max$ operation with addition, 
addition with multiplication, 
subtraction with division, and
the additive identity 0 with the multiplicative identity 1.
Additionally, we replace 1 with a generic fixed constant $\kappa\in\kk$.  The following birational lifts
are the detropicalizations
of the piecewise-linear maps in Definition~\ref{def:trans-PL}.

\begin{defn}[{\cite[\S6]{einpropp}}]\label{def:bir-trans}
Fix a generic constant $\kappa\in\kk$.
We define the following birational maps $\Theta,\nabla,\Delta
:\mathbb{K}^{P}\dashrightarrow\mathbb{K}^{P}$.  
We call $\Theta$ {\bf complementation}, $\down$ {\bf down-transfer},
and $\up$ {\bf up-transfer}.
We again extend $P$ to the poset $\widehat{P}$ as in Definition~\ref{def:trans-PL}.
For all $f\in \kk^P$ and $x\in P$ we set:
\begin{align*}
    (\Theta f)(x) &= \frac{\kappa}{f(x)},\\
    (\down f)(x) &= \frac{f(x)}{\sum\limits_{y\lessdot x} f(y)} \;\;\cblu{\big(\text{with } f\big(\widehat{0}\big)=1\big)},\\
    (\up f)(x) &= 
    \frac{f(x)}{\sum\limits_{y\gtrdot x} f(y)} \;\;\cblu{\big(\text{with } f\big(\widehat{1}\big)=1\big)},\\
    \left(\down^{-1}f\right)(x) &=
    \sum_{\widehat{0}\lessdot y_1\lessdot y_2 \lessdot \cdots \lessdot y_k=x} \kern -25pt f(y_1)f(y_2)\cdots f(y_k)
            =f(x) \sum\limits_{y\lessdot x}\left(\down^{-1}f\right)(y) \;\; \cblu{\big(\text{with } (\down^{-1}f)\big(\widehat{0}\big)=1\big)},\\
    \left(\up^{-1}f\right)(x) &=
    \sum_{x=y_1 \lessdot y_2\lessdot \cdots \lessdot y_k\lessdot \widehat{1}} \kern -25pt f(y_1)f(y_2)\cdots f(y_k)
              =f(x) \sum\limits_{y\gtrdot x}\left(\up^{-1}f\right)(y) \;\; \cblu{\big(\text{with } (\up^{-1}f)\big(\widehat{1}\big)=1\big)}.
\end{align*}
\end{defn}

\begin{defn}[\cite{BAR-motion}]\label{def:rho-BR}
In the birational setting,
we define \textbf{birational order rowmotion} (or
\textbf{BOR-motion}) as $\BOR=\Theta\circ\up^{-1}\circ\down$. 
Similarly,
we define \textbf{birational antichain rowmotion (BAR-motion)} as
$\BAR= \down\circ\Theta\circ\up^{-1}$.  
\end{defn}

We refer the reader to our earlier work~\cite{BAR-motion}
for more detail about BAR-motion.
See Figure~\ref{fig:2x3-BAR} for one iteration of $\BAR$-motion on $[2]\times[3]$.

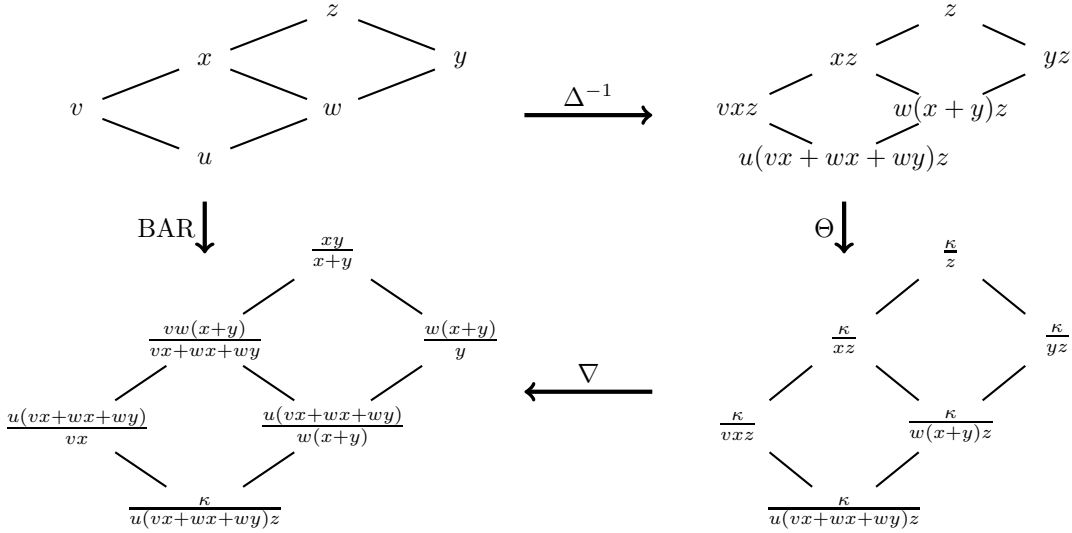
\begin{figure}
\begin{center}
\begin{tikzpicture}[xscale=1.7,yscale=1.15]
\begin{scope}[yscale=4/7]
\node at (0,0) {$u$};
\node at (-1,1) {$v$};
\node at (1,1) {$w$};
\node at (0,2) {$x$};
\node at (2,2) {$y$};
\node at (1,3) {$z$};
\draw[thick] (-0.2,0.2) -- (-0.8,0.8);
\draw[thick] (0.2,0.2) -- (0.8,0.8);
\draw[thick] (1.2,1.2) -- (1.8,1.8);
\draw[thick] (0.2,2.2) -- (0.8,2.8);
\draw[thick] (-0.2,1.8) -- (-0.8,1.2);
\draw[thick] (0.2,1.8) -- (0.8,1.2);
\draw[thick] (1.2,2.8) -- (1.8,2.2);
\end{scope}
\draw[ultra thick, ->] (0,-0.5) -- (0,-1.1);
\node[left] at (0,-0.8) {$\BAR$};
\draw[ultra thick, ->] (2.5,0.5) -- (3.5,0.5);
\node[above] at (3,0.5) {$\up^{-1}$};
\node[above] at (3,-2.7) {$\down$};
\draw[ultra thick, ->] (3.5,-2.7) -- (2.5,-2.7);
\draw[ultra thick, ->] (5,-0.5) -- (5,-1.1);
\node[left] at (5,-0.8) {$\Theta$};
\begin{scope}[yscale=4/7,shift={(5,0)},xscale=5/6]
\node at (0,0) {$u(vx+wx+wy)z$};
\node at (-1,1) {$vxz$};
\node at (1,1) {$w(x+y)z$};
\node at (0,2) {$xz$};
\node at (2,2) {$yz$};
\node at (1,3) {$z$};
\draw[thick] (-0.3,0.3) -- (-0.7,0.7);
\draw[thick] (0.3,0.3) -- (0.7,0.7);
\draw[thick] (1.3,1.3) -- (1.7,1.7);
\draw[thick] (0.3,2.3) -- (0.7,2.7);
\draw[thick] (-0.3,1.7) -- (-0.7,1.3);
\draw[thick] (0.3,1.7) -- (0.7,1.3);
\draw[thick] (1.3,2.7) -- (1.7,2.3);
\end{scope}
\begin{scope}[shift={(0,-4.1)}]
\node at (0,0) {$\frac{\kappa}{u(vx+wx+wy)z}$};
\node at (-1,1) {$\frac{u(vx+wx+wy)}{vx}$};
\node at (1,1) {$\frac{u(vx+wx+wy)}{w(x+y)}$};
\node at (0,2) {$\frac{vw(x+y)}{vx+wx+wy}$};
\node at (2,2) {$\frac{w(x+y)}{y}$};
\node at (1,3) {$\frac{xy}{x+y}$};
\draw[thick] (-0.3,0.3) -- (-0.7,0.7);
\draw[thick] (0.3,0.3) -- (0.7,0.7);
\draw[thick] (1.3,1.3) -- (1.7,1.7);
\draw[thick] (0.3,2.3) -- (0.7,2.7);
\draw[thick] (-0.3,1.7) -- (-0.7,1.3);
\draw[thick] (0.3,1.7) -- (0.7,1.3);
\draw[thick] (1.3,2.7) -- (1.7,2.3);
\end{scope}
\begin{scope}[shift={(5,-4.1)},xscale=5/6]
\node at (0,0) {$\frac{\kappa}{u(vx+wx+wy)z}$};
\node at (-1,1) {$\frac{\kappa}{vxz}$};
\node at (1,1) {$\frac{\kappa}{w(x+y)z}$};
\node at (0,2) {$\frac{\kappa}{xz}$};
\node at (2,2) {$\frac{\kappa}{yz}$};
\node at (1,3) {$\frac{\kappa}{z}$};
\draw[thick] (-0.3,0.3) -- (-0.7,0.7);
\draw[thick] (0.3,0.3) -- (0.7,0.7);
\draw[thick] (1.3,1.3) -- (1.7,1.7);
\draw[thick] (0.3,2.3) -- (0.7,2.7);
\draw[thick] (-0.3,1.7) -- (-0.7,1.3);
\draw[thick] (0.3,1.7) -- (0.7,1.3);
\draw[thick] (1.3,2.7) -- (1.7,2.3);
\end{scope}
\end{tikzpicture}
\end{center}
\caption{One iteration of $\BAR$-motion on $[2]\times[3]$.}
\label{fig:2x3-BAR}
\end{figure}

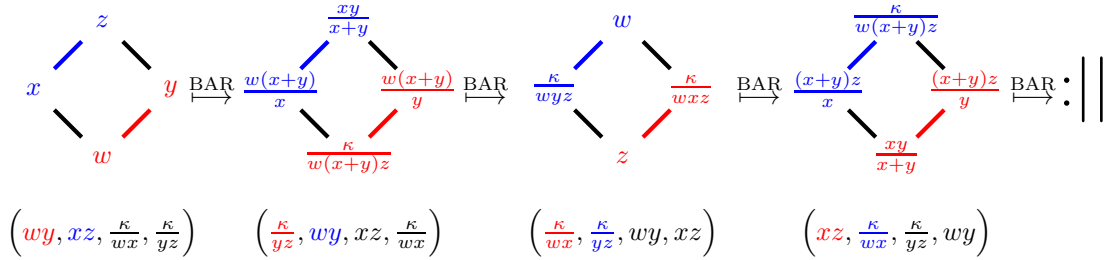
\begin{figure}
    \centering
\begin{tikzpicture}[scale=.91]
\begin{scope}[shift={(.4,0)}]
\draw[blue, ultra thick] (-0.3, 1.7) -- (-0.7, 1.3);
\draw[ultra thick] (0.3, 1.7) -- (0.7, 1.3);
\draw[ultra thick] (-0.7, 0.7) -- (-0.3, 0.3);
\draw[red, ultra thick] (0.7, 0.7) -- (0.3, 0.3);
\node at (0,2) {$ {\color{blue} z  }$};
\node at (-1,1) {$  {\color{blue}x  }$};
\node at (1,1) {$  {\color{red}y } $};
\node at (0,0) {$  {\color{red}w } $};
\node at (1.6,1) {$\stackrel{\BAR}{\longmapsto}$};
\node at (0,-1.1) {$\left({\color{red}wy},{\color{blue}xz},\frac{\kappa}{wx},\frac{\kappa}{yz}\right)$};
\end{scope}
\begin{scope}[shift={(4,0)}]
\draw[blue, ultra thick] (-0.3, 1.7) -- (-0.7, 1.3);
\draw[ultra thick] (0.3, 1.7) -- (0.7, 1.3);
\draw[ultra thick] (-0.7, 0.7) -- (-0.3, 0.3);
\draw[red, ultra thick] (0.7, 0.7) -- (0.3, 0.3);
\node at (0,2) {$ {\color{blue} \frac{xy}{x+y} } $};
\node at (-1,1) {$ {\color{blue} \frac{w(x+y)}{x} } $};
\node at (1,1) {$ {\color{red} \frac{w(x+y)}{y} } $};
\node at (0,0) {$ {\color{red} \frac{\kappa}{w(x+y)z} } $};
\node at (2,1) {$\stackrel{\BAR}{\longmapsto}$};
\node at (0,-1.1) {$\left({\color{red}\frac{\kappa}{yz}},{\color{blue}wy},xz,\frac{\kappa}{wx}\right)$};
\end{scope}
\begin{scope}[shift={(8,0)}]
\draw[blue, ultra thick] (-0.3, 1.7) -- (-0.7, 1.3);
\draw[ultra thick] (0.3, 1.7) -- (0.7, 1.3);
\draw[ultra thick] (-0.7, 0.7) -- (-0.3, 0.3);
\draw[red, ultra thick] (0.7, 0.7) -- (0.3, 0.3);
\node at (0,2) {${\color{blue}  w  }$};
\node at (-1,1) {$ {\color{blue} \frac{\kappa}{wyz}  }$};
\node at (1,1) {$ {\color{red} \frac{\kappa}{wxz}  }$};
\node at (0,0) {${\color{red}  z  }$};
\node at (2,1) {$\stackrel{\BAR}{\longmapsto}$};
\node at (0,-1.1) {$\left({\color{red}\frac{\kappa}{wx}},{\color{blue}\frac{\kappa}{yz}},wy,xz\right)$};
\end{scope}
\begin{scope}[shift={(12,0)}]
\draw[blue, ultra thick] (-0.3, 1.7) -- (-0.7, 1.3);
\draw[ultra thick] (0.3, 1.7) -- (0.7, 1.3);
\draw[ultra thick] (-0.7, 0.7) -- (-0.3, 0.3);
\draw[red, ultra thick] (0.7, 0.7) -- (0.3, 0.3);
\node at (0,2) {$ {\color{blue} \frac{\kappa}{w(x+y)z} } $};
\node at (-1,1) {$ {\color{blue} \frac{(x+y)z}{x}  }$};
\node at (1,1) {$ {\color{red} \frac{(x+y)z}{y}  }$};
\node at (0,0) {$ {\color{red} \frac{xy}{x+y}  }$};
\node at (2,1) {$\stackrel{\BAR}{\longmapsto}$};
\node at (0,-1.1) {$\left({\color{red}xz},{\color{blue}\frac{\kappa}{wx}},\frac{\kappa}{yz},wy\right)$};
\node at (2.75,1) {\Huge $:\hspace{-0.1 in}||$};
\end{scope}
\end{tikzpicture}
    \caption{The full orbit of $\BAR$ on a generic labeling
    for $P=[2]\times[2]$.  Below each labeling is its ST word, illustrating Theorem~\ref{thm:st}.}
    \label{fig:2x2-BAR}
\end{figure}

We detropicalize Definition~\ref{def:ST-PL}
to obtain the birational Stanley--Thomas word.

\begin{defn} \label{def:STword}
Let $a,b\in\zz_{>0}$, $P=[a]\times[b]$, and $g\in\kk^P$.  The birational \textbf{Stanley--Thomas word}
(or \textbf{ST word})
$\st_g$ is the $(a+b)$-tuple
given by
$$\st_g(i)=\left\{\begin{array}{ll}
g(i,1)g(i,2)\cdots g(i,b) &\text{if }1\leq i\leq a,\\
\kappa/\big(g(1,i-a)g(2,i-a)\cdots g(a,i-a)\big) &\text{if }a+1\leq i \leq a+b.\\
\end{array}\right.$$
\end{defn}

\begin{figure}
\centering
\begin{tikzpicture}[xscale=17/9, yscale=1.07]
\begin{scope}
\draw[blue, ultra thick] (-0.3, 1.7) -- (-0.7, 1.3);
\draw[ultra thick] (0.3, 1.7) -- (0.7, 1.3);
\draw[ultra thick] (-0.7, 0.7) -- (-0.3, 0.3);
\draw[red, ultra thick] (0.7, 0.7) -- (0.3, 0.3);
\draw[blue, ultra thick] (0.7, 2.7) -- (0.3, 2.3);
\draw[ultra thick] (1.3, 2.7) -- (1.7, 2.3);
\draw[red, ultra thick] (1.7, 1.7) -- (1.3, 1.3);
\node at (1,3) {${\color{blue}  z  }$};
\node at (0,2) {${\color{blue}  x  }$};
\node at (2,2) {${\color{red}  y  }$};
\node at (-1,1) {${\color{blue}  v  }$};
\node at (1,1) {${\color{red}  w  }$};
\node at (0,0) {${\color{red}  u  }$};
\node[below] at (0.5,-0.45) {$g$};
\node at (0.5,-1.5) {$\st_{g}=\left({\color{red}uwy},{\color{blue}vxz},\frac{\kappa}{uv}, \frac{\kappa}{wx},\frac{\kappa}{yz}\right)$};
\end{scope}
\begin{scope}[shift={(4,0)}]
\draw[blue, ultra thick] (-0.3, 1.7) -- (-0.7, 1.3);
\draw[ultra thick] (0.3, 1.7) -- (0.7, 1.3);
\draw[ultra thick] (-0.7, 0.7) -- (-0.3, 0.3);
\draw[red, ultra thick] (0.7, 0.7) -- (0.3, 0.3);
\draw[blue, ultra thick] (0.7, 2.7) -- (0.3, 2.3);
\draw[ultra thick] (1.3, 2.7) -- (1.7, 2.3);
\draw[red, ultra thick] (1.7, 1.7) -- (1.3, 1.3);
\node at (1,3) {${\color{blue}  \frac{xy}{x+y}  }$};
\node at (0,2) {${\color{blue}  \frac{vw(x+y)}{vx+wx+wy}  }$};
\node at (2,2) {${\color{red}  \frac{w(x+y)}{y}  }$};
\node at (-1,1) {${\color{blue}  \frac{u(vx+wx+wy)}{vx}  }$};
\node at (1,1) {${\color{red}  \frac{u(vx+wx+wy)}{w(x+y)}  }$};
\node at (0,0) {${\color{red}  \frac{\kappa}{u(vx+wx+wy)z}  }$};
\node[below] at (0.5,-0.45) {$\BAR(g)$};
\node at (0.5,-1.5) {$\st_{\BAR(g)}=\left({\color{red}\frac{\kappa}{yz}},{\color{blue}uwy},vxz,\frac{\kappa}{uv}, \frac{\kappa}{wx}\right)$};
\end{scope}
\begin{scope}[shift={(0,-5.5)}]
\draw[blue, ultra thick] (-0.3, 1.7) -- (-0.7, 1.3);
\draw[ultra thick] (0.3, 1.7) -- (0.7, 1.3);
\draw[ultra thick] (-0.7, 0.7) -- (-0.3, 0.3);
\draw[red, ultra thick] (0.7, 0.7) -- (0.3, 0.3);
\draw[blue, ultra thick] (0.7, 2.7) -- (0.3, 2.3);
\draw[ultra thick] (1.3, 2.7) -- (1.7, 2.3);
\draw[red, ultra thick] (1.7, 1.7) -- (1.3, 1.3);
\node at (1,3) {${\color{blue}  \frac{vw}{v+w}  }$};
\node at (0,2) {${\color{blue}  \frac{u(v+w)}{v}  }$};
\node at (2,2) {${\color{red}  \frac{u(v+w)}{w}  }$};
\node at (-1,1) {${\color{blue}  \frac{\kappa}{uwyz}  }$};
\node at (1,1) {${\color{red}  \frac{\kappa}{u(v+w)xz}  }$};
\node at (0,0) {${\color{red}  z  }$};
\node[below] at (0.5,-0.45) {$\BAR^2(g)$};
\node at (0.5,-1.5) {$\st_{\BAR^2(g)}=\left({\color{red}\frac{\kappa}{wx}},{\color{blue}\frac{\kappa}{yz}},uwy,vxz,\frac{\kappa}{uv}\right)$};
\end{scope}
\begin{scope}[shift={(4,-5.5)}]
\draw[blue, ultra thick] (-0.3, 1.7) -- (-0.7, 1.3);
\draw[ultra thick] (0.3, 1.7) -- (0.7, 1.3);
\draw[ultra thick] (-0.7, 0.7) -- (-0.3, 0.3);
\draw[red, ultra thick] (0.7, 0.7) -- (0.3, 0.3);
\draw[blue, ultra thick] (0.7, 2.7) -- (0.3, 2.3);
\draw[ultra thick] (1.3, 2.7) -- (1.7, 2.3);
\draw[red, ultra thick] (1.7, 1.7) -- (1.3, 1.3);
\node at (1,3) {${\color{blue}  u  }$};
\node at (0,2) {${\color{blue}  \frac{\kappa}{uw(x+y)z}  }$};
\node at (2,2) {${\color{red}  \frac{\kappa}{uvxz}  }$};
\node at (-1,1) {${\color{blue}  \frac{(x+y)z}{x}  }$};
\node at (1,1) {${\color{red}  \frac{(x+y)z}{y}  }$};
\node at (0,0) {${\color{red}  \frac{xy}{x+y}  }$};
\node[below] at (0.5,-0.45) {$\BAR^3(g)$};
\node at (0.5,-1.5) {$\st_{\BAR^3(g)}=\left({\color{red}\frac{\kappa}{uv}},{\color{blue}\frac{\kappa}{wx}},\frac{\kappa}{yz},uwy,vxz\right)$};
\end{scope}
\begin{scope}[shift={(0,-11)}]
\draw[blue, ultra thick] (-0.3, 1.7) -- (-0.7, 1.3);
\draw[ultra thick] (0.3, 1.7) -- (0.7, 1.3);
\draw[ultra thick] (-0.7, 0.7) -- (-0.3, 0.3);
\draw[red, ultra thick] (0.7, 0.7) -- (0.3, 0.3);
\draw[blue, ultra thick] (0.7, 2.7) -- (0.3, 2.3);
\draw[ultra thick] (1.3, 2.7) -- (1.7, 2.3);
\draw[red, ultra thick] (1.7, 1.7) -- (1.3, 1.3);
\node at (1,3) {${\color{blue}  \frac{\kappa}{u(vx+wx+wy)z}  }$};
\node at (0,2) {${\color{blue}  \frac{(vx+wx+wy)z}{(v+w)x}  }$};
\node at (2,2) {${\color{red}  \frac{(vx+wx+wy)z}{wy}  }$};
\node at (-1,1) {${\color{blue}  \frac{(v+w)x}{v}  }$};
\node at (1,1) {${\color{red}  \frac{(v+w)xy}{vx+wx+wy}  }$};
\node at (0,0) {${\color{red}  \frac{vw}{v+w}  }$};
\node[below] at (0.5,-0.45) {$\BAR^4(g)$};
\node at (0.5,-1.5) {$\st_{\BAR^4(g)}=\left({\color{red}vxz},{\color{blue}\frac{\kappa}{uv}},\frac{\kappa}{wx},\frac{\kappa}{yz},uwy\right)$};
\end{scope}
\begin{scope}[shift={(4,-11)}]
\draw[blue, ultra thick] (-0.3, 1.7) -- (-0.7, 1.3);
\draw[ultra thick] (0.3, 1.7) -- (0.7, 1.3);
\draw[ultra thick] (-0.7, 0.7) -- (-0.3, 0.3);
\draw[red, ultra thick] (0.7, 0.7) -- (0.3, 0.3);
\draw[blue, ultra thick] (0.7, 2.7) -- (0.3, 2.3);
\draw[ultra thick] (1.3, 2.7) -- (1.7, 2.3);
\draw[red, ultra thick] (1.7, 1.7) -- (1.3, 1.3);
\node at (1,3) {${\color{blue}  z  }$};
\node at (0,2) {${\color{blue}  x  }$};
\node at (2,2) {${\color{red}  y  }$};
\node at (-1,1) {${\color{blue}  v  }$};
\node at (1,1) {${\color{red}  w  }$};
\node at (0,0) {${\color{red}  u  }$};
\node[below] at (0.5,-0.45) {$\BAR^5(g)$};
\node at (0.5,-1.5) {$\st_{\BAR^5(g)}=\left({\color{red}uwy},{\color{blue}vxz},\frac{\kappa}{uv}, \frac{\kappa}{wx},\frac{\kappa}{yz}\right)$};
\end{scope}
\end{tikzpicture}
\caption{An orbit of $\BAR$ starting with a generic labeling $g\in\kk^P$, for $P=[2]\times[3]$.  The order of $\BAR$ on $P$ is $5=2+3$.  The ST word is listed below each labeling.  Note the cyclic shift property of Theorem~\ref{thm:st}.}
\label{fig:BAR[2]x[3]}
\end{figure}
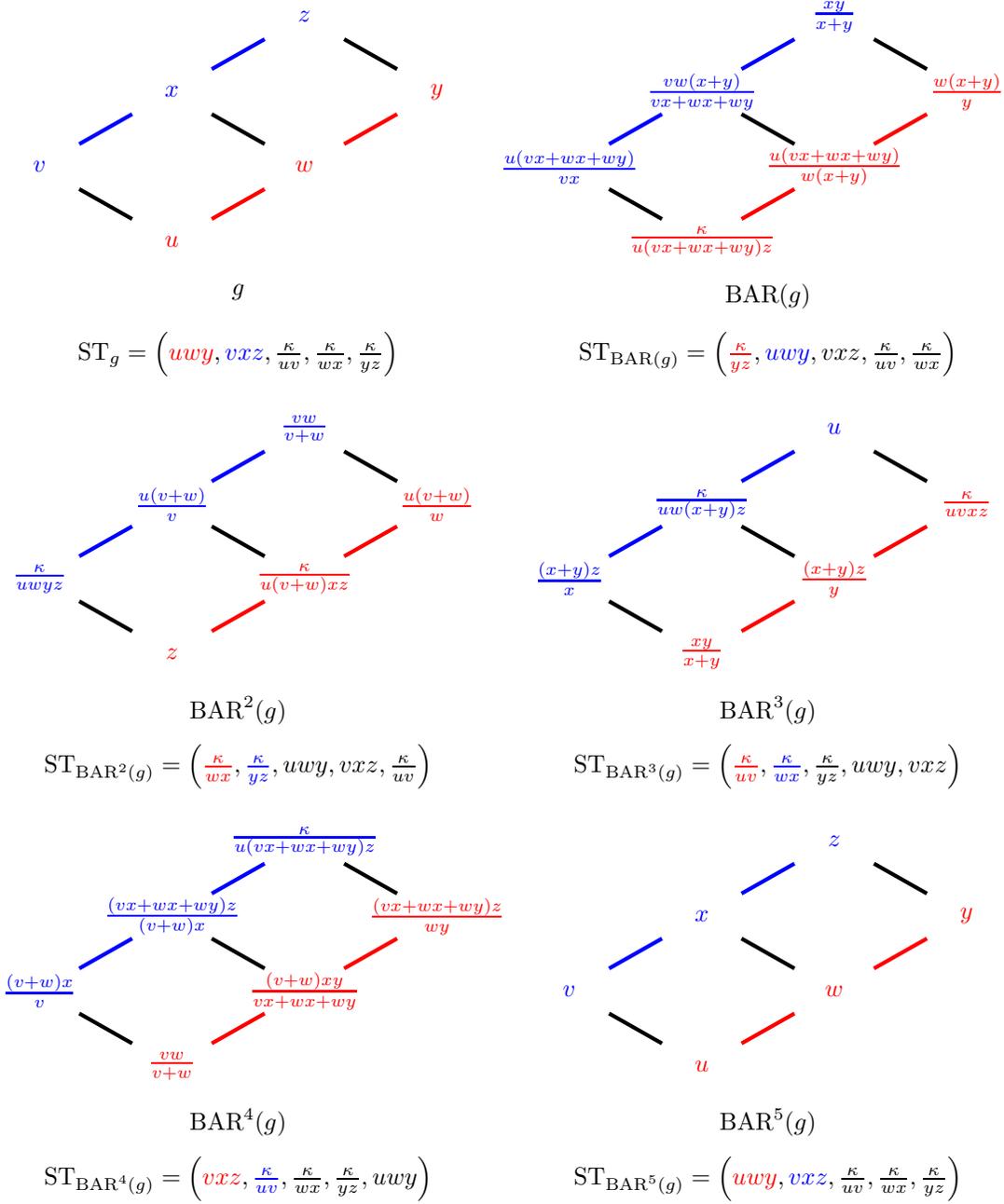

\begin{example}
Let $g$ be the generic labeling of $[2]\times[3]$ displayed in the top left corner of Figure~\ref{fig:BAR[2]x[3]}.  Then
$$\st_g = \big(\st_g(1), \st_g(2), \st_g(3), \st_g(4), \st_g(5)\big)=
\big(uwy, vxz, \kappa/(uv),\kappa/(wx),\kappa/(yz)\big).$$
After applying BAR-motion to $g$,
the Stanley--Thomas word of $\BAR(g)$ is
$$\st_{\BAR(g)}=\big(\kappa/(yz),uwy,vxz,\kappa/(uv),\kappa/(wx) \big)$$ which is simply a rightward cyclic shift
of $\st_g$.
This cyclic shift property is formalized in the following theorem.
\end{example}

\begin{thm}\label{thm:st}
Let $P=[a]\times[b]$.  For a labeling $g\in\kk^P$,
$$\st_{\BAR(g)}(i)=\st_g(i-1) \text{ for } 2\leq i\leq a+b \text{ and }\st_{\BAR(g)}(1)=\st_g(a+b).$$
Thus, $\st_{\BAR^{m}(g)}(i)=\st_g(i-m \bmod a+b)$ for every $i\in \zz$ and $m\in \nn$, where here ``$\bmod \;M$'' returns values within the set $\{1,2,\dots,M\}$. 
\end{thm}

This result follows from its noncommutative analogue (Theorem~\ref{thm:st-nar}), which we
prove in Section~\ref{sec:NClift}.  Any equality of expressions satisfied by a birational
map that does not contain subtraction or additive inverses
also holds in the
piecewise-linear realm (by tropicalization) and furthermore in the combinatorial realm (by
restriction); see~\cite[Remark~10]{GrRo16}.  So Theorem~\ref{thm:st} implies the analogous
statement for the piecewise-linear realm, and thus the already-known result for the
combinatorial realm.  By the end of this paper, we will have a chain of four realms (combinatorial, piecewise-linear,
birational, and noncommutative) and a proof in one realm implies the versions for all
previous realms.



Theorem~\ref{thm:st} is illustrated for $[2]\times[2]$ in Figure~\ref{fig:2x2-BAR} and for $[2]\times[3]$
in Figure~\ref{fig:BAR[2]x[3]}.
Unlike in the combinatorial realm, we cannot use
Theorem~\ref{thm:st} to prove that the order of
$\BAR$ on $[a]\times[b]$ is $a+b$ since the ST word does not uniquely define the labeling;
however, this has been proven by Grinberg and the second author~\cite{GrRo15}
(for birational \emph{order} rowmotion, which implies the order for $\BAR$; see~\cite[Example~3.7]{BAR-motion}).
We can still make use of the ST word to prove
a lifting of the fiber homomesy to the birational realm.
Corollary~\ref{cor:homomesy} has already been proven by Hopkins~\cite[Remark~4.44]{hopkins2019minuscule} using techniques that can be applied to a wider family of posets.
Here we obtain an alternative proof, illustrating that the same technique Propp and the second author used in the combinatorial realm lifts to the birational realm.
In the birational setting, addition has been replaced with multiplication, so a slightly modified definition of homomesy is used, to avoid the $n$th roots in geometric means.
\begin{defn}[{\cite[\S2.1]{einpropp}}]
Let $\cals$ be a collection of combinatorial objects, and $f:\cals \rightarrow \kk$ a
``statistic'' (any map) on $\cals$.
Suppose $\varphi :\cals \ra \cals$ is a map and there exists a positive integer $n$ for
which $\varphi^n$ is the identity on $\cals $. 
Then we say \textbf{$f$ exhibits multiplicative homomesy with respect to $\varphi $} if
$$f(x)f\big(\varphi(x)\big)f\big(\varphi^2(x)\big) \cdots f\big(\varphi^{n-1}(x)\big)$$
is constant, independent of the choice of $x\in \cals $.
\end{defn}

The following follows easily from Theorem~\ref{thm:st}.

\begin{cor}\label{cor:homomesy}
Let $P=[a]\times [b]$, and fix $k \in [a]$ and $\ell \in [b]$. 
Then the  product of labels along the $k$th positive fiber,
$g\mapsto g(k,1) g(k,2) \cdots g(k,b)$,  
exhibits multiplicative homomesy with respect to $\BAR$ because
for each $g\in \kk^{P}$ we have
$$\prod\limits_{m=0}^{a+b-1} (\BAR^m g)(k,1) (\BAR^m g)(k,2) \cdots (\BAR^m g)(k,b)=\kappa^b.$$
Similarly, the product of labels along the $\ell$th negative fiber,
$g\mapsto g(1,\ell) g(2,\ell) \cdots g(a,\ell)$, exhibits multiplicative homomesy because
$$\prod\limits_{m=0}^{a+b-1} (\BAR^m g)(1,\ell) (\BAR^m g)(2,\ell) \cdots (\BAR^m g)(a,\ell)=\kappa^a.$$
\end{cor}

\begin{proof}
First we extend Definition~\ref{def:STword} by setting $\ST_{g}(j)=\ST_{g}(i)$ whenever $j\equiv
i \pmod{a+b}$.  Then by Theorem~\ref{thm:st}, $\ST_{\BAR^{m}g} (i\bmod a+b) = \ST_{g}
(i-m\bmod a+b)$ for any $i\in \zz$.  

Fix $k \in [a]$ and $\ell \in [b]$.
Then
\begin{align*}
\prod\limits_{m=0}^{a+b-1} (\BAR^m g)(k,1) (\BAR^m g)(k,2) \cdots (\BAR^m g)(k,b)
&=
\prod\limits_{m=0}^{a+b-1}\ST_{\BAR^m g}(k)\\
= \prod\limits_{m=0}^{a+b-1}\ST_{g}(k-m \bmod a+b)
&=
\prod_{r=1}^{a+b}\ST_{g}(r)= \kappa^{b}.
\end{align*}
The second equality above is from Theorem~\ref{thm:st}, and the last equality is from the definition
of $\ST_{g}$ since
each element label appears once as a factor of $\prod_{r=1}^{a+b}\ST_{g}(r)$, and once again reciprocated by
$\kappa$.
Now working along negative fibers we similarly obtain
\begin{align*}
\prod\limits_{m=0}^{a+b-1} (\BAR^m g)(1,\ell) (\BAR^m g)(2,\ell) \cdots (\BAR^m g)(a,\ell)
&=
\prod\limits_{m=0}^{a+b-1}\frac{\kappa}{\ST_{\BAR^m g}(\ell+a)}\\
= \prod\limits_{m=0}^{a+b-1}\frac{\kappa}{\ST_{g}(\ell+a-m \bmod a+b)}
&=
\prod_{r=1}^{a+b}\frac{\kappa}{\ST_{g}(r)}= \kappa^{a}.
\end{align*}
\end{proof}

\begin{example}
Consider the product $g \mapsto g(1,1)g(1,2)$ displayed in {\color{red}red} in Figure~\ref{fig:2x2-BAR}.
The product of this statistic, ranging across the entire orbit is
$$\left(w\cdot y\right)\cdot \left(\frac{\kappa}{w(x+y)z}\cdot \frac{w(x+y)}{y}\right) \cdot
\left(z\cdot \frac{\kappa}{wxz}\right) \cdot \left(\frac{xy}{x+y}\cdot \frac{(x+y)z}{y}\right) = \kappa^2.$$
\end{example}

\section{Lifting to the noncommutative realm}\label{sec:NClift}

\subsection{Introduction to noncommutative dynamics}
Our main result, Theorem~\ref{thm:st}, lifts to the noncommutative realm,
where we replace our field $\kk$ with a skew field\footnote{Here a \emph{skew field} or \emph{division ring} is a ring with 1 in which every nonzero element has a multiplicative inverse.  Multiplication in a skew field must be associative, but need not be commutative (the only field axiom not required).}
$\bbs$ of characteristic zero.
This realm was first considered in unpublished work by Grinberg who conjectured that the
periodicity of birational rowmotion on $[a]\times[b]$ (i.e., the order of this map is $a+b$) 
holds even when we consider noncommuting variables.
In earlier work we expanded this by defining the transfer maps and antichain rowmotion in this realm~\cite[\S5]{BAR-motion}.
We always require the generic constant $\kappa\in \bbs$ to be in the
\emph{center} of $\bbs$ (i.e., $\kappa$ commutes with every element of $\bbs$).
We use the term \textbf{partial map} to describe the analogue of a birational map
over skew fields~\cite[Remark~5.5]{BAR-motion}.

\begin{notation}
For greater ease in writing and interpreting rational expressions in the skew field, we write $\overline{x}$ for $x^{-1}$ when $x\in\bbs$.
Also, we use $\prod^\nearrow$ to indicate the indices increase from left to right, and $\prod^\searrow$ to indicate the indices decrease from left to right. 
For example ${\prod\limits_{n=2}^5} ^\nearrow f(n)=f(2)f(3)f(4)f(5)$
while
${\prod\limits_{n=2}^5} ^\searrow f(n)=f(5)f(4)f(3)f(2)$.
\end{notation}

\begin{defn}[{\cite[Definition~5.11]{BAR-motion}}]\label{def:NCtransf}
We define the following noncommutative generalizations of the birational maps of Definition~\ref{def:bir-trans}.
For all $f\in\mathbb{S}^{P}$ and $x\in P$, we
set:
\begin{align*}
    (\Theta f)(x) &= \kappa \cdot \overline{f(x)},\\
    (\down f)(x) &= f(x) \cdot \overline{\sum\limits_{y\lessdot x}f(y)} \;\;\cblu{\big(\text{with } f\big(\widehat{0}\big)=1\big)},\\
    (\up f)(x) &= 
    \overline{\sum\limits_{y\gtrdot x}f(y)} \cdot f(x) \;\; \cblu{\big(\text{with } f\big(\widehat{1}\big)=1\big)},\\
    \left(\down^{-1}f\right)(x) &=
    \sum_{\widehat{0}\lessdot y_1\lessdot y_2 \lessdot \cdots \lessdot y_k=x} \kern -25pt f(y_k)\cdots f(y_2)f(y_1)  
        =f(x) \cdot \sum\limits_{y\lessdot x}\left(\down^{-1}f\right)(y)\;\; \cblu{\big(\text{with } (\down^{-1}f)\big(\widehat{0}\big)=1\big)},\\
    \left(\up^{-1}f\right)(x) &=
    \sum_{x=y_1\lessdot y_2 \lessdot \cdots \lessdot y_k\lessdot \widehat{1}} \kern -25pt  f(y_k)\cdots f(y_2)f(y_1)
          =\sum\limits_{y\gtrdot x}\left(\up^{-1}f\right)(y) \cdot f(x)\;\; \cblu{\big(\text{with } (\up^{-1}f)\big(\widehat{1}\big)=1\big)}.
\end{align*}
\end{defn}

{\bf Noncommutative antichain rowmotion (NAR-motion)} is the partial map $\NAR: \bbs^P \dra  \bbs^P$
given by $\NAR=\down\circ \Theta \circ \up^{-1}$.
However, an equivalent description of NAR-motion (that we will use here in a proof) is in terms of partial maps called toggles.
Toggles have been studied in connection with order-ideal rowmotion since the work of Cameron and Fon-Der-Flaass~\cite{cameronfonder}.
The toggles we will work with here are from combinatorial antichain toggles first
described by Striker~\cite{strikergentog}, then lifted by us to the higher realms
~\cite{antichain-toggling,BAR-motion}.  

\clearpage
\begin{defn}[{\cite[Definition~5.13,~Lemma~5.19]{BAR-motion}}]
Let $v\in P$.  The \textbf{noncommutative antichain toggle} is the partial map $\tau_v:\bbs^P \dra \bbs^P$ defined as follows:

\begin{align*}\big(\tau_v(g)\big)(x) &= 
\left\{\begin{array}{ll}
\kappa \cdot \overline{
\ds \sum_{\widehat{0}\lessdot y_1\lessdot y_2
\lessdot
\cdots \lessdot y_k\lessdot \widehat{1}, \; y_c=v}{\prod\limits_{i=1}^{c-1}} ^\searrow g(y_i) \cdot
{\prod\limits_{i=c}^{k}} ^\searrow g(y_i)
} &\text{if $x=v$}\\\vspace{-4.5pt}\\
g(x) &\text{if $x\not=v$}
\end{array}
\right.\\
&=
\left\{
\begin{array}{ll}
\kappa \cdot \overline{(\up^{-1}g)(v)} \cdot \overline{(\down^{-1}g)(v)}\cdot g(v)
 &\text{if $x=v$}\\\vspace{-4.5pt}\\
g(x) &\text{if $x\not=v$.}
\end{array}
\right.
\end{align*}
\end{defn}

Let $(x_1,x_2,\dots,x_n)$ be any linear extension of a finite poset $P$.  Then $\NAR$ can be equivalently defined as $\NAR=\tau_{x_n}\cdots \tau_{x_2} \tau_{x_1}$, i.e., toggling at each element of $P$
from bottom to top; the equivalence is proved in~\cite[Theorem 5.26]{BAR-motion}.
We will make use of both the transfer map and toggle descriptions of $\NAR$.

\subsection{Noncommutative Stanley--Thomas word}

\begin{figure}
\centering
\begin{tikzpicture}[xscale=17/9, yscale=15/18]
\begin{scope}
\draw[thick] (0.3, 1.7) -- (0.7, 1.3);
\draw[thick] (0.7, 2.7) -- (0.3, 2.3);
\draw[thick] (1.3, 2.7) -- (1.7, 2.3);
\draw[thick] (1.7, 1.7) -- (1.3, 1.3);
\node at (1,3) {$  z  $};
\node at (0,2) {$  x  $};
\node at (2,2) {$  y  $};
\node at (1,1) {$  w  $};
\node[below] at (1,0.55) {$g=\NAR^4(g)$};
\node at (1,-0.5) {$\st_{g}=\left(yw, zx,
\kappa \cdot\overline{w}\cdot\overline{x}, \kappa \cdot\overline{y}\cdot\overline{z} \right)$};
\end{scope}
\begin{scope}[shift={(4.4,0)}]
\draw[thick] (0.3, 1.7) -- (0.7, 1.3);
\draw[thick] (0.7, 2.7) -- (0.3, 2.3);
\draw[thick] (1.3, 2.7) -- (1.7, 2.3);
\draw[thick] (1.7, 1.7) -- (1.3, 1.3);
\node at (1,3) {$  \overline{\big( \overline{x}+\overline{y} \big)}  $};
\node at (0,2) {$  \overline{x} \cdot (x+y) \cdot w  $};
\node at (2,2) {$  \overline{y} \cdot (x+y) \cdot w  $};
\node at (1,1) {$  \kappa \cdot \overline{w} \cdot \overline{(x+y)} \cdot \overline{z}  $};
\node[below] at (1,0.55) {$\NAR(g)$};
\node at (1,-0.5) {$\st_{\NAR(g)}=\left(\kappa \cdot\overline{y}\cdot\overline{z},
yw, zx,
\kappa \cdot\overline{w}\cdot\overline{x} \right)$};
\end{scope}
\begin{scope}[shift={(0,-4.3)}]
\draw[thick] (0.3, 1.7) -- (0.7, 1.3);
\draw[thick] (0.7, 2.7) -- (0.3, 2.3);
\draw[thick] (1.3, 2.7) -- (1.7, 2.3);
\draw[thick] (1.7, 1.7) -- (1.3, 1.3);
\node at (1,3) {$  w  $};
\node at (0,2) {$  \kappa \cdot \overline{w} \cdot \overline{y} \cdot \overline{z}  $};
\node at (2,2) {$  \kappa \cdot \overline{w} \cdot \overline{x} \cdot \overline{z}  $};
\node at (1,1) {$  z  $};
\node[below] at (1,0.55) {$\NAR^2(g)$};
\node at (1,-0.5) {$\st_{\NAR^2(g)}=\left(\kappa \cdot\overline{w}\cdot\overline{x},
\kappa \cdot\overline{y}\cdot\overline{z}, yw, zx
 \right)$};
\end{scope}
\draw (-.7,-.9) -- (7, -.9);
\begin{scope}[shift={(4.4,-4.3)}]
\draw[thick] (0.3, 1.7) -- (0.7, 1.3);
\draw[thick] (0.7, 2.7) -- (0.3, 2.3);
\draw[thick] (1.3, 2.7) -- (1.7, 2.3);
\draw[thick] (1.7, 1.7) -- (1.3, 1.3);
\node at (1,3) {$  \kappa \cdot \overline{w} \cdot \overline{(x+y)} \cdot \overline{z}  $};
\node at (0,2) {$  z\cdot(x+y)\cdot \overline{x}  $};
\node at (2,2) {$  z\cdot(x+y)\cdot \overline{y}  $};
\node at (1,1) {$  \overline{\big( \overline{x}+\overline{y} \big)}  $};
\node[below] at (1,0.55) {$\NAR^3(g)$};
\node at (1,-0.5) {$\st_{\NAR^3(g)}=\left(zx, \kappa \cdot\overline{w}\cdot\overline{x},
\kappa \cdot\overline{y}\cdot\overline{z}, yw
 \right)$};
\end{scope}
\end{tikzpicture}
\caption{The NAR-orbit for the generic labeling on $P=[2]\times[2]$.  The order of $\NAR$ is 4.  Below each labeling is its Stanley--Thomas word, illustrating Theorem~\ref{thm:st-nar}.}
\label{fig:NAR[2]x[2]-orbit}
\end{figure}
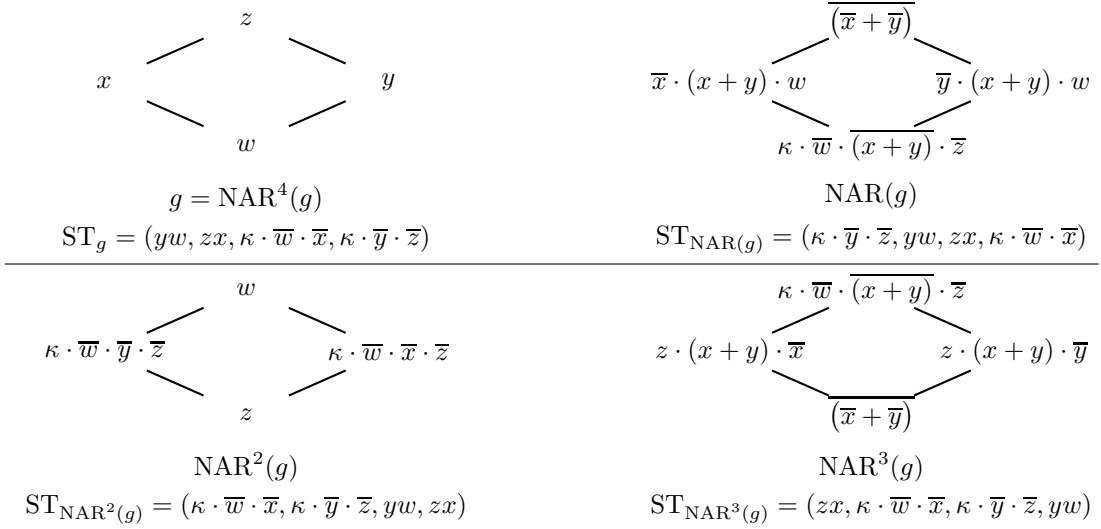

\begin{defn}
Let $a,b\in\zz_{>0}$, $P=[a]\times[b]$, and $g\in\bbs^P$ where $\bbs$ is a skew field.  The \textbf{Stanley--Thomas word} $\st_g$ is the $(a+b)$-tuple
given by
$$\st_g(i)=\left\{\begin{array}{ll}
g(i,b)\cdot g(i,b-1) \cdots g(i,1) &\text{if }1\leq i\leq a,\\
\kappa \cdot \overline{g(1,i-a)} \cdot \overline{g(2,i-a)} \cdots \overline{g(a,i-a)}
&\text{if }a+1\leq i \leq a+b.\\
\end{array}\right.$$
\end{defn}

\begin{thm}\label{thm:st-nar}
Let $P=[a]\times[b]$.  For a labeling $g\in\bbs^P$,
\[
\st_{\NAR(g)}(i)=\st_g(i-1) \text{ for } 2\leq i\leq a+b \text{ and } \st_{\NAR(g)}(1)=\st_g(a+b).
\]
\end{thm}

\begin{example}\label{eg:NAR2x2}
Figure~\ref{fig:NAR[2]x[2]-orbit} shows the generic NAR-orbit for $P=[2]\times[2]$.  Observe that NAR cyclically shifts the Stanley--Thomas word.  We remind the reader that $\kappa$ commutes with every element of $\bbs$, but
in general simplifications in skew fields can be rather tricky. For example $\overline{\overline{x} + \overline{y}}$ can equivalently be written as
\begin{itemize}
    \item $y\overline{(x+y)}x$ by multiplying on the left by $y\overline{y}$ and the right by $\overline{x}x$
    and using the property $\overline{AB}=\overline{B}\cdot
    \overline{A}$,
    \item or as $x\overline{(x+y)}y$ by multiplying on the left by $x\overline{x}$ and the right by $\overline{y}y$,
\end{itemize}
but is {\bf not equivalent} to
$yx\overline{(x+y)}$, $\overline{(x+y)}xy$, $xy\overline{(x+y)}$, or $\overline{(x+y)}yx$. Such identities are necessary even to check the equality $\st_{\NAR(g)}(2)=yw$.   
\end{example}

The following is used in the proof of Theorem~\ref{thm:st-nar}. 

\begin{thm}\label{thm:describe-nar-axb}
Let $P=[a]\times[b]$ and $g\in \bbs^P$.
Let $G=\up^{-1}g$.
Then
\begin{itemize}
    \item $(\NAR g)(1,1) = \kappa \cdot\overline{G(1,1)}$,
    \item $(\NAR g)(1,j) = \overline{G(1,j)} \cdot {G(1,j-1)}$
    for $j\geq 2$,
    \item $(\NAR g)(i,1) = \overline{G(i,1)} \cdot{G(i-1,1)}$
    for $i\geq 2$,
    \item $(\NAR g)(i,j)= 
    \overline{G(i,j)} \cdot G(i-1,j) \cdot g(i-1,j-1)
    \cdot \overline{G(i-1,j-1)} \cdot G(i,j-1)$\\
    \phantom{11111111111l}
    $=\overline{G(i,j)} \cdot G(i,j-1) \cdot g(i-1,j-1)
    \cdot \overline{G(i-1,j-1)} \cdot G(i-1,j)$
    \\for $i,j\geq 2$.
\end{itemize}
\end{thm}

\begin{remark}\label{rem:reduce-formulas}
In Theorem~\ref{thm:describe-nar-axb},
we can also use the formulas in the fourth bullet point when $i=1$ and/or $j=1$ if we define
$G(i,j) = 1$
when $i=0$ and/or $j=0$, $g(0,0)=\kappa$, and
$g(i,j)=1$ when one of $i,j$ (but not both) are 0.
However, there are other ways to define $g(i,j)$ and $G(i,j)$ in these out-of-bounds positions that would also serve this same purpose of reducing the number of different formulas.
We will not choose a convention here, in case
a specific one makes the most sense in further study of NAR-motion.
\end{remark}

We will actually prove
the following theorem which is more general than Theorem~\ref{thm:describe-nar-axb} and applies to a wider range of posets.

\begin{thm}~\label{thm:describe-nar-axb-more-general}
Let $P$ be a finite poset, and let
$g\in\bbs^P$.  Let $G=\up^{-1}g$.  Let $x\in P$.
\begin{enumerate}
    \item If $x$ is a minimal element of $P$, then
    $(\NAR g)(x)=\kappa \cdot \overline{G(x)}$.
    \item If $x$ covers exactly one element $z\in P$, then
    $(\NAR g)(x)=\overline{G(x)}\cdot G(z)$.
    \item If $u,v,z\in P$ satisfy $u\not=v$ and
    $\{y\in P: y\lessdot x\} = \{y\in P: y\gtrdot z\} = \{u,v\}$,
    then
    $$(\NAR g)(x)=\overline{G(x)}\cdot G(u)
    \cdot g(z) \cdot \overline{G(z)} \cdot G(v).$$
\end{enumerate}
\end{thm}

\begin{proof}
Let $H=(\Theta\circ\up^{-1})(g)$.
So $\NAR g = \down H$, and for each $e\in P$, $H(e)=\kappa\cdot \overline{G(e)}$.

We begin with (1).  Since $x$ is a minimal element of $P$, $(\NAR g)(x)=(\down H)(x)=H(x)
=\kappa\cdot \overline{G(x)}$.

For (2), we have
\begin{align*}
    (\NAR g)(x)&= (\down H)(x)\\
    &= H(x) \cdot \overline{H(z)}\\
    &= \kappa\cdot \overline{G(x)} \cdot \overline{\kappa\cdot \overline{G(z)}}\\
    &= \overline{G(x)}\cdot G(z).
\end{align*}

Now we consider (3).
From the definition of $\up^{-1}$, we have $G(z)=(G(u)+G(v))g(z)$.  This is used in the penultimate equality below:
\begin{align*}
    (\NAR g)(x)&= (\down H)(x)\\
    &= H(x)\cdot\big(\overline{H(u)+H(v)}\big)\\
    &= \kappa\cdot \overline{G(x)}\cdot
    \big(\overline{ \kappa\cdot \overline{G(u)} +
    \kappa\cdot \overline{G(v)} }\big)\\
    &=\overline{G(x)}\cdot \big(\overline{ \overline{G(u)} + \overline{G(v)} }\big)\\
    &= \overline{G(x)}\cdot G(u) \cdot \big(\overline{G(u)+G(v)}\big) \cdot G(v)\\
    &= \overline{G(x)}\cdot G(u) \cdot \big(\overline{G(z)\cdot \overline{g(z)}}\big) \cdot G(v)\\
    &= \overline{G(x)}\cdot G(u)
    \cdot g(z) \cdot \overline{G(z)} \cdot G(v).
\end{align*}
The identity $\overline{\overline{x} + \overline{y}} = x\overline{(x+y)}y$ from Example~\ref{eg:NAR2x2} was used in the fifth equality.
\end{proof}

Now we will use Theorem~\ref{thm:describe-nar-axb} to prove Theorem~\ref{thm:st-nar}.

\begin{proof}[of Theorem~\ref{thm:st-nar}]
The theorem follows easily from the following four equalities:

\begin{enumerate}
    \item ${\prod\limits_{\ell=1}^b}^\searrow (\NAR g)(1,\ell) = \kappa\cdot{\prod\limits_{k=1}^a}^\nearrow \overline{g(k,b)}$,
    \item ${\prod\limits_{\ell=1}^b}^\searrow (\NAR g)(k,\ell) = {\prod\limits_{\ell=1}^b}^\searrow g(k-1,\ell)$ for $2\leq k\leq a$,
    \item ${\prod\limits_{k=1}^a}^\searrow (\NAR g)(k,1) = \kappa\cdot{\prod\limits_{\ell=1}^b}^\nearrow \overline{g(a,\ell)}$,
    \item ${\prod\limits_{k=1}^a}^\searrow (\NAR g)(k,\ell) = {\prod\limits_{k=1}^a}^\searrow g(k,\ell-1)$ for $2\leq \ell\leq b$.
\end{enumerate}
We will prove only the first two, since the last two are respectively analogous but with the coordinates reversed.
Let $G=\up^{-1}g$.
For the first equality we apply Theorem~\ref{thm:describe-nar-axb} and unravel the resulting telescoping
product to obtain 
\begin{align*}
    {\prod\limits_{\ell=1}^b}^\searrow (\NAR g)(1,\ell) &=
    \left({\prod\limits_{\ell=2}^b}^\searrow
    \overline{G(1,\ell)}\cdot G(1,\ell-1)\right)
    \cdot \kappa \cdot \overline{G(1,1)}
    =\kappa\cdot\overline{G(1,b)}\\
    &=\kappa\cdot\overline{{\prod\limits_{k=1}^a}^\searrow g(k,b)}
    =\kappa\cdot{\prod\limits_{k=1}^a}^\nearrow \overline{g(k,b)}.
\end{align*}
The third equality above is because there is only one saturated chain from $(1,b)$ to $\widehat{1}$.
Now for the second equality, thanks to Theorem~\ref{thm:describe-nar-axb} and a telescoping product we have
\begin{align*}
    &\phantom{=\;}{\prod\limits_{\ell=1}^b}^\searrow (\NAR g)(k,\ell)\\
    &=
    \left({\prod\limits_{\ell=2}^b}^\searrow
    \overline{G(k,\ell)} \cdot G(k-1,\ell) \cdot g(k-1,\ell-1)\cdot \overline{G(k-1,\ell-1)} \cdot G(k,\ell-1) \right)\\
    &\phantom{=\;}\cdot \overline{G(k,1)} \cdot G(k-1,1)\\
    &= \overline{G(k,b)}\cdot G(k-1,b) \cdot {\prod\limits_{\ell=2}^b}^\searrow g(k-1,\ell-1)\\
    &= g(k-1,b) \cdot {\prod\limits_{\ell=1}^{b-1}}^\searrow g(k-1,\ell)\\
    &= {\prod\limits_{\ell=1}^b}^\searrow g(k-1,\ell)
\end{align*}
where the third equality follows from the fact that $g=\up G$ and a reindexing of the product.
\end{proof}

\section{Future directions}\label{sec:future}
There are several directions for future research building off of this work.
First of all, the description of how $\NAR$ acts on each of the labels (Theorem~\ref{thm:describe-nar-axb}) may be one piece of the puzzle
in proving Grinberg's conjecture that $\NAR$ has order $a+b$
on the poset $[a]\times[b]$.

The cyclic shifting action on the Stanley--Thomas word has order $a+b$, but this does not prove periodicity of $\NAR$ as
$g\mapsto \ST_g$ is not injective.
However, what we have established in this work is \emph{resonance} as defined by Dilks, Pechenik, and
Striker~\cite{dpsresonance,dilks-striker-vorland}.
We believe there are more examples
of resonance to be found lurking in the birational (or noncommutative) realm.
A direction for further study is to see if the resonance on various posets, such as products $[a]\times[b]\times[c]$ of three chains
considered in~\cite{dpsresonance} generalizes to the birational realm.
The birational realm may also prove useful in answering
unsolved resonance conjectures
in the combinatorial realm.

When restricting Theorem~\ref{thm:describe-nar-axb}
to commutative variables, we obtain the following.  (Again, Remark~\ref{rem:reduce-formulas} applies.)

\begin{thm}\label{thm:describe-bar-axb}
Let $P=[a]\times[b]$ and $g\in \kk^P$.  Then
\begin{itemize}
    \item $(\BAR g)(1,1) = \dfrac{\kappa}{(\up^{-1}g)(1,1)}$,
    \item $(\BAR g)(1,j) = \dfrac{(\up^{-1}g)(1,j-1)}{(\up^{-1}g)(1,j)}$
    for $j\geq 2$,
    \item $(\BAR g)(i,1) = \dfrac{(\up^{-1}g)(i-1,1)}{(\up^{-1}g)(i,1)}$
    for $i\geq 2$,
    \item $(\BAR g)(i,j) = \dfrac{(\up^{-1}g)(i-1,j)\cdot (\up^{-1}g)(i,j-1)\cdot g(i-1,j-1)}{(\up^{-1}g)(i-1,j-1)\cdot
    (\up^{-1}g)(i,j)}$
    for $i,j\geq 2$.
\end{itemize}
\end{thm}

This theorem describes how $\BAR$ acts on each of the individual labels.  We believe this is a first step toward one of our goals: to give
a nice description of how $\BAR$ acts on all of the factors that arise.  For example, on $P=[2]\times[3]$,
the factors $x+y$ and $vx+wx+wy$ arise when applying $\BAR$ (and $v+w$ appears after applying $\BAR$ a second time).
We would like to better understand what $\BAR$ does to these factors as well.  This should help us come up with a formula
for iterations of $\BAR$, i.e., $\big(\BAR^k g\big)(i,j)$
similar to the one found by Musiker and the second author for
birational \emph{order} rowmotion~\cite{musiker-roby}.

Another direction for further study is to study other posets.
Sam Hopkins and the first author have used the birational Stanley--Thomas word to prove a homomesy result on the type $A$ root poset
$\Phi^+(A_n)$~\cite{birational-LK} by using an embedding due to Grinberg and the second author of $\Phi^+(A_n)$ into $[n+1]\times[n+1]$ described in~\cite[Lemma~67]{GrRo15}.
Though this work is in the commutative birational realm, we also conjecture periodicity for noncommutative rowmotion on $\Phi^+(A_n)$.

It is quite possible that there may be analogues in other posets; $[a]\times[b]$ is in the larger family of minuscule posets.
Homomesy and periodicity of rowmotion has been found in all minuscule posets, even at the birational realm~\cite{okada-minuscule}.
Computer data suggests noncommutative rowmotion is periodic on minuscule posets as well as on root posets of coincidental types
(see~\cite[\S8]{hppw} for a definition), and trapezoid posets.
Trapezoid posets
$$T_{a,b} = \{(i,j)\in\zz^2: 1\leq i\leq a \text{ and } i\leq j\leq a+b-i \}$$
for $a\leq b$ were first considered by
Stembridge~\cite{stembridge-trapezoid}
and Stanley~\cite{stanley-plane-partitions}.
The poset $T_{a,b}$ is a doppelg\"anger of
$[a]\times[b]$ in the language of Hamaker,
Patrias, Pechenik, and Williams~\cite{hppw},
meaning $T_{a,b}$ and $[a]\times[b]$ have the same order polynomial.
Trapezoid posets have sparked increased interest
recently in~\cite{reiner-tenner-yong,hopkins2019minuscule} and appear to share similar
properties to rectangle posets, so there may be an analogue to this work within $T_{a,b}$.
We conjecture that rowmotion has order $a+b$
on $T_{a,b}$ in the birational and noncommutative realms, but this is yet to be proved even in the birational realm.
Periodicity in the combinatorial realm
has been proven only very recently by constructing an equivariant bijection between
rowmotion on $[a]\times[b]$ and rowmotion on $T_{a,b}$~\cite{minnesota-reu-trapezoid}.

\acknowledgements
\label{sec:ack}
The authors gratefully acknowledge useful conversations with 
David Einstein,
Darij Grinberg,
Sam Hopkins,
Gregg Musiker,
S\={o}ichi Okada,
James Propp,
Vic Reiner,
Jessica Striker,
Hugh Thomas,
and Nathan Williams. In particular, Grinberg made a number of helpful suggestions on an
earlier draft of this paper, including shortening the proof of one of our main results.  The first author is grateful for travel support provided by a Simons collaboration grant awarded to James Propp.
Several computations invaluable
to this work were done in Sage~\cite{sage}.

\nocite{*}

\begin{thebibliography}{DWYZ20}

\bibitem[AST13]{ast}
D.~Armstrong, C.~Stump, and H.~Thomas.
\newblock A uniform bijection between nonnesting and noncrossing partitions.
\newblock {\em Transactions of the American Mathematical Society},
  365(8):4121--4151, 2013.
\newblock Also available as \arxiv{1101.1277v2}.

\bibitem[BS74]{brouwer1974period}
A.~Brouwer and L.~Schrijver.
\newblock On the period of an operator, defined on antichains.
\newblock {\em Stichting Mathematisch Centrum. Zuivere Wiskunde}, (ZW
  24/74):1--13, 1974.

\bibitem[CF95]{cameronfonder}
P.~Cameron and D.~Fon{-}Der{-}Flaass.
\newblock Orbits of antichains revisited.
\newblock {\em European J. Combin.}, 16(6):545--554, 1995.

\bibitem[DPS17]{dpsresonance}
K.~Dilks, O.~Pechenik, and J.~Striker.
\newblock Resonance in orbits of plane partitions and increasing tableaux.
\newblock {\em Journal of Combinatorial Theory, Series A}, 148:244--274, 2017.
\newblock Also available at \arxiv{1512.00365v2}.

\bibitem[DSV19]{dilks-striker-vorland}
K.~Dilks, J.~Striker, and C.~Vorland.
\newblock Rowmotion and increasing labeling promotion.
\newblock {\em Journal of Combinatorial Theory, Series A}, 164:72--108, 2019.
\newblock Also available at \arxiv{1710.07179v3}.

\bibitem[DWYZ20]{minnesota-reu-trapezoid}
Q.~V. Dao, J.~Wellman, C.~Yost{-}Wolff, and S.~W. Zhang.
\newblock Rowmotion orbits of trapezoid posets.
\newblock {\em Preprint}, 2020.
\newblock \arxiv{2002.04810v1}.

\bibitem[EP18]{einpropp}
D.~Einstein and J.~Propp.
\newblock Combinatorial, piecewise-linear, and birational homomesy for products
  of two chains.
\newblock {\em Preprint}, 2018.
\newblock \arxiv{1310.5294v3}.

\bibitem[GR15]{GrRo15}
D.~Grinberg and T.~Roby.
\newblock Iterative properties of birational rowmotion {II}: rectangles and
  triangles.
\newblock {\em Electron. J. Combin.}, 22(3):Paper 3.40, 49, 2015.

\bibitem[GR16]{GrRo16}
D.~Grinberg and T.~Roby.
\newblock Iterative properties of birational rowmotion {I}: generalities and
  skeletal posets.
\newblock {\em Electron. J. Combin.}, 23(1):Paper 1.33, 40, 2016.

\bibitem[HJ20]{birational-LK}
S.~Hopkins and M.~Joseph.
\newblock The birational {L}alanne-{K}reweras involution.
\newblock {\em Preprint}, 2020.
\newblock \arxiv{2012.15795v1}.

\bibitem[Hop20]{hopkins2019minuscule}
S.~Hopkins.
\newblock Minuscule doppelg{\"a}ngers, the coincidental down-degree
  expectations property, and rowmotion.
\newblock {\em Experimental Mathematics}, pages 1--29, 2020.
\newblock Also available at \arxiv{1902.07301v3}.

\bibitem[HPPW20]{hppw}
Z.~Hamaker, R.~Patrias, O.~Pechenik, and N.~Williams.
\newblock Doppelg{\"a}ngers: bijections of plane partitions.
\newblock {\em International Mathematics Research Notices}, 2020(2):487--540,
  2020.
\newblock Also available at \arxiv{1602.05535v3}.

\bibitem[Jos19]{antichain-toggling}
M.~Joseph.
\newblock Antichain toggling and rowmotion.
\newblock {\em Electron J. Combin.}, 26(1), 2019.
\newblock Also available at \arxiv{1709.09331v3}.

\bibitem[JR20]{BAR-motion}
M.~Joseph and T.~Roby.
\newblock Birational and noncommutative lifts of antichain toggling and
  rowmotion.
\newblock {\em Algebraic Combinatorics}, 3(4):955--984, 2020.
\newblock Also available at \arxiv{1909.09658v3}.

\bibitem[Kir01]{kirillov2001introduction}
A.N. Kirillov.
\newblock Introduction to tropical combinatorics.
\newblock In {\em Physics and combinatorics}, pages 82--150. World Scientific,
  2001.

\bibitem[MR19]{musiker-roby}
G.~Musiker and T.~Roby.
\newblock Paths to understanding birational rowmotion on products of two
  chains.
\newblock {\em Algebraic Combinatorics}, 2(2):275--304, 2019.
\newblock Also available at \arxiv{1801.03877v3}.

\bibitem[NY04]{noumi-yamada}
M.~Noumi and Y.~Yamada.
\newblock Tropical {R}obinson-{S}chensted-{K}nuth correspondence and birational
  {W}eyl group actions.
\newblock In {\em Representation Theory of Algebraic Groups and Quantum
  Groups}, pages 371--442, Tokyo, Japan, 2004. Mathematical Society of Japan.

\bibitem[Oka21]{okada-minuscule}
S.~Okada.
\newblock Birational rowmotion and {C}oxeter-motion on minuscule posets.
\newblock {\em Electron J. Combin.}, 28(1), 2021.
\newblock Also available at \arxiv{2004.05364v1}.

\bibitem[Pan09]{panyushev}
D.~Panyushev.
\newblock On orbits of antichains of positive roots.
\newblock {\em European J. Combin.}, 30(2):586--594, 2009.
\newblock Also available at \arxiv{0711.3353v2}.

\bibitem[PR15]{propproby}
J.~Propp and T.~Roby.
\newblock Homomesy in products of two chains.
\newblock {\em Electron. J. Combin.}, 22(3), 2015.
\newblock Also available at \arxiv{1310.5201v5}.

\bibitem[RTY18]{reiner-tenner-yong}
V.~Reiner, B.~E. Tenner, and A.~Yong.
\newblock Poset edge densities, nearly reduced words, and barely set-valued
  tableaux.
\newblock {\em Journal of Combinatorial Theory, Series A}, 158:66--125, 2018.
\newblock Also available at \arxiv{1603.09589v3}.

\bibitem[S{\etalchar{+}}19]{sage}
W.~Stein et~al.
\newblock {\em {S}age {M}athematics {S}oftware ({V}ersion 8.8)}.
\newblock The Sage Development Team, 2019.
\newblock \url{http://www.sagemath.org}.

\bibitem[Sta86a]{stanley-plane-partitions}
R.~Stanley.
\newblock Symmetries of plane partitions.
\newblock {\em Journal of Combinatorial Theory, Series A}, 43(1):103--113,
  1986.

\bibitem[Sta86b]{Sta86}
R.~Stanley.
\newblock Two poset polytopes.
\newblock {\em Discrete \& Computational Geometry}, 1(1):9--23, 1986.

\bibitem[Sta09]{Sta09}
R.~Stanley.
\newblock Promotion and evacuation.
\newblock {\em Electron. J. Combin.}, 16(2, Special volume in honor of Anders
  Bj\"{o}rner):Research Paper 9, 24, 2009.

\bibitem[Sta11]{ec1ed2}
R.~Stanley.
\newblock {\em Enumerative combinatorics, volume 1, 2nd edition}.
\newblock Cambridge University Press, 2011.
\newblock Also available at \url{http://math.mit.edu/~rstan/ec/ec1/}.

\bibitem[Ste86]{stembridge-trapezoid}
J.~Stembridge.
\newblock Trapezoidal chains and antichains.
\newblock {\em European Journal of Combinatorics}, 7(4):377--387, 1986.

\bibitem[Str18]{strikergentog}
J.~Striker.
\newblock Rowmotion and generalized toggle groups.
\newblock {\em Discrete Mathematics \& Theoretical Computer Science}, 20, 2018.
\newblock Also available at \arxiv{1601.03710v5}.

\bibitem[SW12]{strikerwilliams}
J.~Striker and N.~Williams.
\newblock Promotion and rowmotion.
\newblock {\em European J. Combin.}, 33:1919--1942, 2012.
\newblock Also available at \arxiv{1108.1172v3}.

\end{thebibliography}
\newcommand{\etalchar}[1]{$^{#1}$}

\label{sec:biblio}

\end{document}